\RequirePackage{fix-cm}
\documentclass[smallextended]{svjour3}       % onecolumn (second format)
\smartqed  % flush right qed marks, e.g. at end of proof
\usepackage{graphicx}
\usepackage{amsfonts}
\usepackage{amsmath}
\usepackage{mathrsfs}
\usepackage{cite}
\newcommand{\ucite}[1]{{\cite{#1}}}
\newtheorem{thm}{Theorem}
  \newtheorem{defn}[thm]{Definition}
  \newtheorem{lem}[thm]{Lemma}
  
  \newtheorem{prop}[thm]{Proposition}
  \newtheorem{rem}[thm]{Remark}

\begin{document}

\title{Regularity of solutions to time-harmonic Maxwell's system with various lower than Lipschitz coefficients \thanks{The research of Basang Tsering-xiao was supported in part by the National Natural Science Foundation of China(NSFC)grant (11261054), and partly by the Fok Ying Tung Education Foundation grant (151102). The research of Wei Xiang was supported in part by the Research Grants Council of the HKSAR, China
(Project No. CityU 21305215, Project No. CityU 11332916, Project No. CityU
11304817 and Project No. CityU 11303518).}
}

\titlerunning{Regularity of solutions with low regularity coefficients}        % if too long for running head

\author{Basang Tsering-xiao     \and
        Wei Xiang %etc.
}


\institute{B. Tsering-xiao \at
              Mathematical Department of Tibet University, Lhasa, Tibet, P.R.China. 850000 \\
             % School of Mathematical Sciences, Peking University, No 5. Yiheyuan Road, Beijing, P.R.China 10087 \\
              Tel: +86-0891-6405210\\
              \email{basangtu@qq.com}           %  \\
           \and
           W. Xiang \at
              Department of Mathematics, City University of Hong Kong, Kowloon, Hong Kong, P.R. China.\\
             \email{weixiang@cityu.edu.hk}  
}

\date{Received: date / Accepted: date}

\maketitle

\begin{abstract}
In this paper, we study the regularity of the solutions of Maxwell's equations in a bounded domain. We consider several different types of low regularity assumptions to the coefficients which are all less than Lipschitz. We first develop a new approach by giving $\mathcal{H}^1$ estimate when the coefficients are $\mathcal{L}^{\infty}$ bounded; and then we derive $\mathcal{W}^{1,p}$ estimates for every $p > 2$ when one of the leading coefficients is simply continuous; Finally, we extend the result to $\mathcal{C}^{1,\alpha}$ almost everywhere for the solution of the homogeneous Maxwell's equations when the coefficients are $\mathcal{W}^{1,p}, \, p>3$ and close to the identity matrix in the sense of $\mathcal{L}^{\infty}$ norm.
 The last two estimates are new, and the techniques and methods developed here can also be applied to other problems with similar difficulties.
\end{abstract}

\section{Introduction}
\label{intro}

Let $\Omega \subset \mathbb{R}^3$ be an open bounded domain with $C^{1,1}$ boundary.
The time-harmonic electromagnetic field $(E,H)\in \mathcal{H}_{\textrm{loc}}\left(\textrm{curl};\Omega\right)$,
where for any domain $\Omega$ we define
$$
\mathcal{H}_{\textrm{loc}}\left(\textrm{curl};\Omega \right):=
\left\{
u\in \mathcal{L}_{\textrm{loc}}^{2}\left(\Omega\right)
\textrm{ such that }
\nabla\wedge u\in \mathcal{L}_{\textrm{loc}}^{2}\left(\Omega\right)
\right\},
$$
satisfies Maxwell's equations

\begin{equation} \label{eq0}
\begin{cases}
\nabla\wedge{E}  -i\,\omega\mu_{0}\mu(x)\,{H}
                         & = J_m \textrm{ in }\Omega\,,\\
\nabla\wedge{H}  +i\,\omega\varepsilon_{0}\varepsilon(x)\,{E}
                         & = J_e \textrm{ in }\Omega\,, \\
E\wedge\boldsymbol{\nu}  &= G \wedge \boldsymbol{\nu}  \textrm{ on }\partial\Omega\,,
\end{cases}
\end{equation}
where $\varepsilon$ and $\mu$ are real matrix-valued functions in
$\mathcal{L}^{\infty}\left(\mathbb{R}^{3}\right)^{3\times3}$, the terms $J_m,J_e \in \mathcal{L}^2(\Omega)$ are current sources and the boundary condition $G$ is in $\mathcal{H}_{\textrm{loc}}(\textrm{curl};\Omega)$.
We assume that $\varepsilon^{-1}(x)$ and $\mu^{-1}(x)$ are real, uniformly positive definite and bounded, that is , there exist $0<\lambda_1\le\lambda_2<\infty$ and $0<\lambda_3\le\lambda_4<\infty$ such that for all $\xi\in \mathbb{R}^3$ and almost every $x \in \mathbb{R}^3$,
\begin{equation}
\begin{aligned}
\lambda_1|\xi|^2\le\varepsilon^{-1}(x)\xi \cdot \xi  & \le \lambda_2 |\xi|^2,  \label{cond1}\\
\lambda_3|\xi|^2\le\mu^{-1}(x)\xi \cdot \xi & \le \lambda_4 |\xi|^2.
\end{aligned}
\end{equation}

The goal of this paper is to study the regularity of the electromagnetic fields with low regularity assumptions on the material parameters. The material parameters $(\varepsilon(x), \mu(x))$ which are the matrix-valued coefficients of Maxwell's equations, and the importance of the low regularity assumptions on the coefficients lies in consideration of the corresponding medium with complicated structure, such as liquid crystals.  This research topic is originally motivated by the study of the electromagnetic inverse problems for liquid crystals, see \ucite{basang} in which we study the molecular structure of liquid crystals using the optical measurement which is represented by the solutions of Maxwell's equations. The question of whether the data measured is well understood in the sense of the object observed is of paramount importance in this context.
 %for liquid crystals. See, in which Maxwell's equations describe how the light propagates in liquid crystals medium, and the authors studied the reconstruction of the coefficients through optical measurements. The coefficients of Maxwell's equations in \ucite{basang} represent the molecular structure of nematic liquid crystals. The optical measurements which we can obtain through experiment is a formula of the solutions of Maxwell's equations.
%Therefore it is important to know the relations between the coefficients and the solution for obtaining a better understanding of the objects we study, and the questions investigated usually have two aspects as follows:
%\begin{itemize}
%\item How does the solution rely on the coefficients (Forward problem).
%\item How much information of the coefficients can be reconstructed from a limited knowledge of the solution (Inverse problem).
%\end{itemize}

There is some work have been done recently in this consideration;  Capdeboscq and Tsering-xiao\ucite{ct} studied the reconstruction of the coefficients of anisotropic Maxwell's equations using the optical measurement of a laser beam.  Okaji \ucite{ot} and Colton \ucite{cd} studied the uniqueness of the solutions of Maxwell's equations when the coefficients are smooth, and Nguyen and Wang \ucite{ntwj}, and Ball, Capdeboscq and Tsering-xiao{\ucite{bct}} studied the uniqueness when the coefficients hold low regularity assumptions.

Many people, see{\ucite{vgpr, pm, gw,y,mssv}}, have studied the regularity of the solutions of Maxwell's equations. In the case of the high regularity assumptions on the coefficients, Leis\ucite{lei} established well-posedness in $\mathcal{H}^1(\Omega)$ when the coefficients are smooth matrices where $\mathcal{H}^1(\Omega)$ is the Sobolev space and later Costabel \ucite{cost} and Fernandes \ucite{pfmo} showed the $\mathcal{C}^1$ regularity for $\mathcal{C}^{1,1}$ domain. However, Alberti and Capdeboscq\ucite{gayc} is the only work so far that studied the regularity of the solutions with low regularity assumptions on the coefficients. In this paper, we extend the regularity results of \ucite{gayc} in the case of real coefficients. Specifically, we work on the system of Maxwell's equations instead of studying the deduced scalar elliptic equations of second order. We first develop the approach through giving the $\mathcal{H}^1$ estimates for the solutions when the coefficients are bounded; Based on the $\mathcal{H}^1$ estimates, we derive $\mathcal{W}^{1,p}$ estimates of $H$ for every $p > 2$  when the coefficients $\varepsilon$ is simply continuous and $\mu \in \mathcal{W}^{1,p}$ for some $ p > 3$; Finally, we obtain that the solution $H$ of the homogeneous Maxwell's equations is $\mathcal{C}^{1+\alpha}$ almost everywhere when $\varepsilon \in C^{\beta}$, $\mu \in \mathcal{W}^{1,p}, \, p>3$ , and the difference between $\mu$ and the identity matrix is small in the sense of $\mathcal{L}^{\infty}$ norm. %The $C^{1+\alpha}$ estimate is one of the most important results in this paper, which is the best regularity result for the solutions so far when the regularity of the coefficients are less than %.

Based on the regularity results of $H$ and taking into account the symmetric structure of the equations, we can easily derive the interior estimates of the solution $E$. However, the boundary conditions of $E$ and $H$ are different, so it is necessary to use a different approach to derive the boundary estimates for $E$. By developing a different approach for the boundary estimates of $E$, we obtain the same regularity for $E$ and $H$ if we give the same assumptions on the coefficients $\varepsilon$ and $\mu$ respectively.

Section 3 gives the proofs of the results given in Section 2. Each of the proofs includes of two main parts, of which the first part consists of the proofs corresponding to the homogeneous case, and the second part is the proof of each theorem given in Section 2.

\section{Main results}

We investigate the regularity of the solutions of the equation (\ref{eq0}) with inhomogeneous boundary conditions. Without loss of the generality, we will first study the regularity of the solution $H$, since the regularity of $E$ inside the domain can be studied similarly.
%following problem
%\begin{equation}\label{eq:maxwell-system-inhomo}
%\begin{cases}
%\nabla\wedge{E}  -i\,\omega\mu_{0}\mu(x)\,{H}
%                        = J_m & \textrm{ in }\Omega\,,\\
%\nabla\wedge{H}  +i\,\omega\varepsilon_{0}\varepsilon(x)\,{E}
  %                       = J_e & \textrm{ in }\Omega\,,
%\end{cases}
%\end{equation}
%where $\varepsilon(x)$ and $\mu(x)$ are real matrix-valued functions in
%$\mathcal{L}^{\infty}\left(\mathbb{R}^{3}\right)^{3\times3}$. Without loss of generality, we let $E\wedge\boldsymbol{\nu}=G\wedge\boldsymbol{\nu}$ on $\partial\Omega$ and derive the interior estimates.
By \eqref{eq0}, $H$ satisfies the following equation  

\begin{equation}\label{equ:3.2}
\nabla\wedge(\varepsilon^{-1}(x)\nabla\wedge H)=k_1^2\mu(x)H-k_2J_m+\nabla\wedge(\varepsilon^{-1}(x)J_e)\,,
\end{equation}
and
\begin{equation}\label{equ:3.3}
{\rm div}(\mu(x)H)=k_3{\rm div} J_m\,,
\end{equation}
where the constants are given by $k_1^2=\omega^2\varepsilon_0\mu_0$, $k_2=i\omega\varepsilon_0$, and $k_3=i/(\omega\mu_0)$.  The boundary conditions of $H$ can be written as follows
\begin{equation}\label{equ:2.3x}
(\varepsilon^{-1}\nabla\wedge H)\wedge \boldsymbol{\nu}=(\varepsilon^{-1}J_e)\wedge\boldsymbol{\nu}-k_2 G\wedge\boldsymbol{\nu}\qquad\mbox{on }\partial\Omega.
\end{equation}
Then by applying (3.52) in \cite{pm}, {\it i.e.}, 
\begin{equation}
{\rm div}_{\partial\Omega}(E\wedge\boldsymbol{\nu})=(\nabla\wedge E)\cdot\boldsymbol{\nu}, 
\end{equation}
then we can get the second boundary condition for $ H $,
\begin{equation}\label{equ:2.4x}
(\mu H)\cdot\boldsymbol{\nu}=k_3J_m\cdot\boldsymbol{\nu}-k_3(\nabla\wedge G)\cdot\boldsymbol{\nu} \qquad\mbox{on }\partial\Omega.
\end{equation}

%In the inhomogeneous case, we obtain the same results as we achieved in the homogeneous case, and we address each of the results in the similar order as in the previous context.

\begin{thm}[$\mathcal{H}^1$ regularity]\label{thm:jejm}
Let $\mu(x)\in \mathcal{W}^{1,3+\delta}(\Omega)^{3\times 3}$, where $\delta>0$. Suppose $\varepsilon^{-1}(x) \in \mathcal{L}^{\infty}(\Omega)^{3\times 3}$ and satisfies \eqref{cond1},$J_m \in {\mathcal{W}^{1,2}(\Omega)}$ $J_e \in \mathcal{L}^2(\Omega)$ and $G \in \mathcal{H}_{\textrm{loc}}(\textrm{curl};\Omega)$. If $H$ is a weak solution of equations \eqref{equ:3.2} and \eqref{equ:3.3}, with the boundary conditions \eqref{equ:2.3x} and \eqref{equ:2.4x}, then $H\in \mathcal{H}^1(\Omega)$ and
$$
\begin{aligned}
||H||_{\mathcal{H}^{1}(\Omega)}
\leq& C(||\nabla\wedge G||_{\mathcal{L}^2(\Omega)}+||J_m||_{\mathcal{W}^{1,2}(\Omega)}\\
 &~~~+||H||_{\mathcal{L}^2(\Omega)}+||J_e||_{\mathcal{L}^2(\Omega)}+||(J_m-\nabla\wedge G)\cdot\boldsymbol{\nu}||_{\mathcal{H}^{\frac{1}{2}}(\partial\Omega)}).
\end{aligned}
$$
\end{thm}

\begin{rem}
The result of Theorem \ref{thm:jejm} is same with the result addressed in Theorem 1 in \ucite{gayc}; however, we derive the result by developing a new approach. We obtain the subsequent new regularity results given in the theorems below based on the approach developed in this theorem.
\end{rem}
Based on the $\mathcal{H}^{1}$-regularity, by applying the Stampacchia interpolation theorem, we derive the $\mathcal{W}^{1,p}$ estimate of $H$ for every $p \geq 2$ when the coefficient $\varepsilon^{-1}(x)$ is simply continuous.

\begin{thm}[$\mathcal{W}^{1,p}$ regularity]\label{thm:inw1p}
Let $\mu(x)\in \mathcal{W}^{1,3+\delta}(\Omega)^{3\times 3}$, where $\delta>0$. Suppose $\varepsilon^{-1}(x)\in \mathcal{C}^0(\Omega)^{3\times3}$ is a positive definite matrix,  $J_m \in {\mathcal{W}^{1,p}(\Omega)}$ $J_e \in \mathcal{L}^p(\Omega)$ and $G \in \mathcal{H}_{\textrm{loc}}(\textrm{curl};\Omega)$. If $H$ is a weak solution of equations \eqref{equ:3.2} and \eqref{equ:3.3}, with the boundary conditions \eqref{equ:2.3x} and \eqref{equ:2.4x}, then $H\in \mathcal{W}^{1,p}(\Omega)$ for every $p \geq 2$ and
$$
\begin{aligned}
||H||_{\mathcal{W}^{1,p}(\Omega)}
\leq& C(||\nabla\wedge G||_{\mathcal{L}^p(\Omega)}+||J_m||_{\mathcal{W}^{1,p}(\Omega)}\\
 &~~~+||H||_{\mathcal{L}^2(\Omega)}+||J_e||_{\mathcal{L}^p(\Omega)}+||(J_m-\nabla\wedge G)\cdot\boldsymbol{\nu}||_{\mathcal{W}^{1-\frac{1}{p},q}(\partial\Omega)}),
\end{aligned}
$$
where $\frac{1}{p}+\frac{1}{q}=1$.
\end{thm}

\begin{rem}
To produce the same regularity result, we assume that $\varepsilon^{-1}(x)$ is simply continuous and $\mu(x)\in \mathcal{W}^{1,3+\delta}(\Omega)^{3\times 3}$ while Theorem 3 in \ucite{gayc} assumes that both $\mu(x)$ and $\varepsilon(x)$ are in $\mathcal{W}^{1,3+\delta}(\Omega)^{3\times 3}$.
\end{rem}

Now we extend the results and give $\mathcal{C}^{1,\alpha}$-regularity of $H$ almost everywhere provided that $\varepsilon^{-1}(x)$ is H\"{o}lder continuous.
\begin{thm}[$\mathcal{C}^{1,\alpha}$ regularity]\label{thm:15}
Let $\mu(x)\in \mathcal{W}^{1,3+\delta}(\Omega)^{3 \times 3}$, where $\delta>0$. Suppose there exits $\delta_0>0$ sufficiently small such that $||\mu(x)-I||_{\mathcal{L}^{\infty}(\Omega)^{3 \times 3}} < \delta_0$ where $I$ is the $3\times3$ identity matrix, and $k_1$ is not a Maxwell eigenvalue. Suppose $\varepsilon^{-1}(x)\in \mathcal{C}^{\beta}(\Omega)^{3\times3}$, $0< \beta < 1$, and $H$ is a weak solution of equations \eqref{equ:3.2} and \eqref{equ:3.3} with $J_m=J_e=0$, then there exists an open set $\Omega_h\subset\Omega$ and $\alpha \in (0,1)$, such that $meas(\Omega\backslash\Omega_h)=0$,     and for any $x_0\in\Omega_h$, there exists $r>0$ such that
$$
\begin{aligned}
||H||_{\mathcal{C}^{1,\alpha}(B_r(x_0))}
\leq& C(1
 +||H||_{\mathcal{L}^2(\Omega)}),
\end{aligned}
$$
%+||J_m||_{\mathcal{C}^{1,\alpha}(\Omega)}+||J_e||_{\mathcal{C}^{\alpha}(\Omega)}
where $B_r(x_0)$ is the ball centering at $x_0$ with radius $r>0$.
\end{thm}
%Since we only need to repeat the proof of theorem \ref{thm:6} and the \emph{Step 3} of the proof of theorem \ref{thm:9mu} by adding the corresponding source terms, we omit the details for the .
Here we say $k_1$ is a Maxwell eigenvalue if and only if there exist an integer $j$ and a point $x_0\in\Omega$, such that $k_1=e_j$, where $e_j$ is the j$^{\text{th}}$ Maxwell eigenvalue guaranteed by Theorem 4.18 of the book \cite{pm} (which is written as $k_j$ in \cite{pm}), with the coefficients $\varepsilon\equiv\varepsilon(x_0)$ and $\mu\equiv\mu(x_0)$.

\begin{rem}
 The results obtained in Theorem \ref{thm:inw1p} and Theorem \ref{thm:15} are new regularity results for the solution, and the low regularity assumption on the coefficients are important for many application problems, i.e. electromagnetic inverse problems for liquid crystals in which the permeability $\mu$ is the identity matrix and permittivity $\varepsilon$ is necessary to have a low regularity property, see \cite{basang} for more detail.
\end{rem}

\begin{rem}
Theorem \ref{thm:15} can not be extended to the inhomogeneous equations, since we found a counter example for the inhomogeneous equations. One of the simple counterexample can be constructed when $H=(1+\sigma V(x_1),0,0), \mu=\frac{1}{(1+\sigma V(x_1))}I, \varepsilon=I, J_e=(0,0,0), J_m =-i\omega(1,0,0)$, where $V(x_1)$ is the Volterra's function, $\sigma$ is any small positive constant, and $I$ is the $3\times3$ identity matrix.
\end{rem}

Similar to the regularity properties of $H$ given above, we also derive the regularity properties for the electrical field $E$ when the coefficients are given similar assumptions. 

\begin{thm}[Regularity for $E$]\label{thm:E4i}Let $\varepsilon(x)\in \mathcal{W}^{1,3+\delta}(\Omega)^{3 \times 3}$, where $\delta>0$. Suppose the condition \eqref{cond1} holds, and if $E$ is a weak solution of the equations
\eqref{eq0}, then
$E\in \mathcal{H}^1(\Omega)$ and
$$||E||_{\mathcal{H}^1(\Omega)}\le C(||E||_{\mathcal{L}^2(\Omega)}+||\nabla\wedge G||_{\mathcal{L}^2(\Omega)}+||J_e||_{\mathcal{W}^{1,2}(\Omega)}+||J_m||_{\mathcal{L}^2(\Omega)}).$$
Moreover, if $\mu^{-1}(x)\in \mathcal{C}^0(\Omega)^{3\times3}$ and satisfies the condition \eqref{cond1},
then the following inequality holds
\begin{equation}
||E||_{\mathcal{W}^{1,p}(\Omega)}\leq C(||E||_{\mathcal{L}^2(\Omega)}+||\nabla\wedge G||_{\mathcal{L}^p(\Omega)}+||J_e||_{\mathcal{W}^{1,p}(\Omega)}+||J_m||_{\mathcal{L}^{p}(\Omega)}),
\end{equation}
where $p \geq 2$ and the constant $C$ does not depend on the solution $E$;
Suppose $\mu^{-1}(x)\in \mathcal{C}^{\beta}(\Omega)^{3 \times 3}$, and $k_1$ is not a Maxwell eigenvalue; If there exits $\delta_0>0$ small such that $||\varepsilon(x)-I||_{\mathcal{L}^{\infty}(\Omega)^{3 \times 3}} < \delta_0$, and exist an open set $\Omega_h\subset\Omega$ and $\alpha \in (0,1)$, such that $meas(\Omega\backslash\Omega_h)=0$, then for any $x_0\in\Omega_h$ there exists $r>0$ such that the solution $E$ of the homogeneous Maxwell's equations $(J_m=J_e=0)$  satisfies that
\begin{equation}
||E||_{\mathcal{C}^{1,\alpha}(B_r(x_0))}\le C(1+||E||_{\mathcal{L}^2(\Omega)}) %+||J_e||_{\mathcal{C}^{1,\alpha}(\Omega)}+||J_m||_{\mathcal{C}^{\alpha}(\Omega)}),
\end{equation}
where $B_r(x_0)$ is the ball centering at $x_0$ with radius $r>0$.
\end{thm}

\begin{rem}
Clearly, if we exchange the role of $\varepsilon(x)$ and $\mu(x)$, then the arguments given in the theorems above can be easily applied to derive the regularity of $E$. However $E$ and $H$ hold different boundary conditions in Theorem \ref{thm:jejm} and Theorem \ref{thm:inw1p}, hence we use different methods to prove the boundary estimate of $E$ .
\end{rem}

We combine all the previous results together such that if $\mu(x)$ and $\varepsilon(x)$ are regular enough, then we can have the following general results.
\begin{thm}[Regularity for $(E,H)$]\label{them:10a}
Let $\mu(x), \varepsilon(x)\in \mathcal{W}^{1,3+\delta}(\Omega)^{3\times 3}$, where $\delta>0$. Let $\varepsilon(x)$ and $\mu(x)$ satisfy \eqref{cond1}.
Then the solution $(E,H)$ satisfies the estimate that
$$
\begin{aligned}
||H||_{\mathcal{W}^{1,p}(\Omega)}+&||E||_{\mathcal{W}^{1,p}(\Omega)}
\leq C(||\nabla\wedge G||_{\mathcal{L}^p(\Omega)}+||J_m||_{\mathcal{W}^{1,p}(\Omega)}+||J_e||_{\mathcal{W}^{1,p}(\Omega)}\\
 &+||E||_{\mathcal{L}^2(\Omega)}+||H||_{\mathcal{L}^2(\Omega)}+||(J_m-\nabla\wedge G)\cdot\boldsymbol{\nu}||_{\mathcal{W}^{1-\frac{1}{p},q}(\partial\Omega)}),
\end{aligned}
$$
where $\frac{1}{p}+\frac{1}{q}=1$.
Moreover, we assume $k_1$ is not Maxwell eigenvalue, and there exits $\delta_0>0$ small such that $||\mu(x)-I||_{\mathcal{L}^{\infty}(\Omega)^{3 \times 3}}+||\varepsilon(x)-I||_{\mathcal{L}^{\infty}(\Omega)^{3 \times 3}} < \delta_0$. Then there exists an open set $\Omega_h\subset\Omega$ and $\alpha \in (0,1)$, such that $meas(\Omega\backslash\Omega_h)=0$, and for any $x_0\in\Omega_h$, there exists $r>0$ such that, the solutions $E$ and $H$ of the homogeneous Maxwell's equations  $(J_m=J_e=0)$  satisfy the following estimates
\begin{equation}
||E||_{\mathcal{C}^{1,\alpha}(B_r(x_0))}+||H||_{\mathcal{C}^{1,\alpha}(B_r(x_0))}
\le C(1+||E||_{\mathcal{L}^2(\Omega)}+||H||_{\mathcal{L}^2(\Omega)}
),\nonumber
\end{equation}
%+||J_e||_{\mathcal{C}^{1,\alpha}(\Omega)}+||J_m||_{1,\mathcal{C}^{\alpha}(\Omega)}
where $B_r(x_0)$ is the ball centering at $x_0$ with radius $r>0$.
\end{thm}
The proofs of the theorems are given in section \ref{sec:4a}.
%The proof of theorem \uline{\ref{thm:E4i}--\ref{them:10a}} \bs{\ref{them:10a}}is similar in spirit to the method used, as well as the technics developed to deal with the boundary and source terms in the proof of the theorem \uline{\ref{thm:jejm}--\ref{thm:15}} \bs{\ref{thm:jejm}--\ref{thm:E4i}}, hence we omit the details here for brevity.
%\bs{The proof of theorem \ref{thm:E4i} uses similar method which used in the proofs of the previous theorems for $H$ apart from the boundary estimates for $E$, hence we give proposition }

%%%%%%%%%%%%%%%%%%%%%%%%%%%%%%%%%%%%%%
%%%%%%%%%%%% New section %%%%%%%%%%%%%
%%%%%%%%%%%%%%%%%%%%%%%%%%%%%%%%%%%%%%
%%%%%%%%% Proof of the theorems %%%%%%
%%%%%%%%%%%%%%%%%%%%%%%%%%%%%%%%%%%%%%

\section{Proof of the theorems}\label{sec:4a}
\subsection{Proof of Theorem \ref{thm:jejm}}
 In this part, we first study the simpler version of Theorem \ref{thm:jejm}, namely the regularity property of the solution of the homogeneous equations with homogeneous boundary conditions and then we extend the result to the case of Theorem \ref{thm:jejm}. In the homogeneous case, as shown in Proposition \ref{thm:1}, the coefficients are prescribed the same conditions as Theorem \ref{thm:jejm} and we study the derivation of the result in detail. We then give the shorter version of the proof when we apply the approach to the case of Theorem \ref{thm:jejm} and closely study the different part of the argument caused by the inhomogeneous terms and inhomogeneous boundary conditions. 

%We first give $\mathcal{H}^1$ estimate of the solutions $H$ when the coefficient $\varepsilon^{-1}$ is simply bounded.
\begin{prop}\label{thm:1}
Let $\mu(x)\in \mathcal{W}^{1,3+\delta}(\Omega)^{3 \times 3}$, where $\delta>0$.
Suppose $\varepsilon^{-1}(x) \in \mathcal{L}^{\infty}(\Omega)^{3\times 3}$ and that the condition \eqref{cond1} holds. If $H$ is a weak solution of the following equations
\begin{equation}\label{eq1}
\nabla\wedge(\varepsilon^{-1}(x)\nabla\wedge H)=k^2\mu(x)H, \qquad\mbox{in }\Omega
\end{equation}
with the boundary condition $(\varepsilon^{-1}\nabla\wedge H)\wedge \boldsymbol{\nu}=0$ on $\partial\Omega$,
then $H\in \mathcal{H}^1(\Omega)$ and
$$||H||_{\mathcal{H}^1(\Omega)}\le C||H||_{\mathcal{L}^2(\Omega)}.$$
\end{prop}

\begin{proof}[Proof of Proposition \ref{thm:1}] We first show the estimate of $\nabla\wedge H$ which can be proved by the simple energy estimate as follows.

\medskip
{\bf Step 1.} \emph{Estimates of $\nabla\wedge H$.}
Multiply both sides of the equation \eqref{eq1} by $\overline{H}$, and integrate it by parts, we get that
$$\int_{\Omega}\varepsilon^{-1}(x)(\nabla\wedge H)\cdot(\overline{\nabla\wedge H})dx+\int_{\partial\Omega}((\varepsilon^{-1}\nabla\wedge H)\wedge\boldsymbol{\nu})\cdot \overline{H} d\sigma=k^2\int_{\Omega}(\mu H)\cdot \overline{H}dx.$$
Notice that $(\varepsilon^{-1}\nabla\wedge H)\wedge\boldsymbol{\nu}=E\wedge\boldsymbol{\nu}=0$ on $\partial\Omega$, and by applying the condition \eqref{cond1}, we obtain that
\begin{equation}\label{curlH}
\int_{\Omega}|\nabla\wedge H|^2dx \le C\int_{\Omega}|H|^2dx.
\end{equation}

\medskip
{\bf Step 2.} \emph{Helmholtz decomposition}.
From Lemma \ref{lem16} which listed in the Appendix ( see for more detail of the lemma in Amrouche, Seloula \cite{acsn} and Amrouche, Bernardi, Dauge \cite{accdm}), we know that $H$ can be decomposed as $H=\nabla\varphi+\nabla\wedge A$. Notice that $\nabla\wedge(\nabla\wedge A)=\nabla\wedge H$, and ${\rm div}(\nabla\wedge A)=0$, so
$$
\nabla\wedge A\in \mathcal{W}^{1,2}(\Omega)\, ,
$$
provided if $\nabla\wedge H\in \mathcal{L}^2(\Omega)$. More precisely, we have the estimates
\begin{equation}\label{curlA}
||\nabla\wedge A||_{\mathcal{W}^{1,2}(\Omega)}\leq C(||\nabla\wedge H||_{\mathcal{L}^2(\Omega)}+||H||_{\mathcal{L}^2(\Omega)})\, .
\end{equation}

Next from the homogeneous case of the equation \eqref{eq0},  we get that
$$
i \omega\mu_0{\rm div}(\mu\nabla\varphi+\mu\nabla\wedge A)=-{\rm div}(\nabla\wedge E)=0 \,,
$$
hence $\varphi$ satisfies the following elliptic equation
\begin{equation}\label{equ:varphi3.1}
{\rm div}(\mu\nabla\varphi)=-{\rm div}(\mu\nabla\wedge A)\qquad\mbox{in }\Omega \, .
\end{equation}
The boundary condition \eqref{equ:2.4x} can be deduced as the following,
$$
0=(\mu H)\cdot\boldsymbol{\nu}=(\mu\nabla\varphi+\mu\nabla\wedge A)\cdot\boldsymbol{\nu} \,.
$$
Then $\varphi$ satisfies
\begin{equation}\label{bou:varphi3.2}
(\mu\nabla\varphi)\cdot\boldsymbol{\nu}=-(\mu\nabla\wedge A)\cdot\boldsymbol{\nu}\qquad\mbox{on }\partial\Omega \,.
\end{equation}
Multiply $\overline\varphi$ on both sides of the equation \eqref{equ:varphi3.1} and integrate it by part,
$$
\int_{\partial\Omega}\overline{\varphi}(\mu\nabla\varphi)\cdot\boldsymbol{\nu}{\rm d}\sigma-\int_{\Omega}\mu\nabla\varphi\cdot\overline{\nabla\varphi}{\rm d}x=-\int_{\partial\Omega}\overline{\varphi}(\mu\nabla\wedge A)\cdot\boldsymbol{\nu}{\rm d}\sigma+\int_{\Omega}\mu\nabla\wedge A\cdot\overline{\nabla\varphi}{\rm d}x \,.
$$
By employing the boundary condition \eqref{bou:varphi3.2}, and the H\"{o}lder inequality, we have that
\begin{equation}\label{gradphi}
%||\varphi||_{\mathcal{W}^{1,2}(\Omega)}\leq C
||\nabla\varphi||_{\mathcal{L}^2(\Omega)}\leq C||\mu\nabla\wedge A||_{\mathcal{L}^2(\Omega)}\leq C||\nabla\wedge H||_{\mathcal{L}^2(\Omega)}\,.
\end{equation}

\medskip

{\bf Step 3}. \emph{Interior estimates}.
Let $\eta$ be a smooth positive cut-off function such that $\eta\equiv1$ in $B_r$ and $\eta=0$ on $\partial B_R$, where $B_r=B_r(x_0)$ and $B_R=B_R(x_0)$ with $r<R$  and $B_R(x_0)\subset\Omega$. Then take derivatives $\partial_{k}$ on both sides of the equation \eqref{equ:varphi3.1} and multiply $\eta^2\partial_k\overline{\varphi}$, where $k=1,2,3$. After an integration by part, it gives that
\begin{align*}
\lambda_1\int_{B_R}& \eta^2|\partial_k\nabla\varphi|^2{\rm d}x\leq \int_{B_R}\eta^2\mu\partial_k\nabla\varphi\cdot\overline{\partial_k\nabla\varphi}{\rm d}x\\
=&-\int_{B_R}2\eta\overline{\partial_k\varphi}\mu\nabla\eta\cdot\partial_k\nabla\varphi{\rm d}x-\int_{B_R}\eta^2\partial_k\mu\nabla\varphi\cdot\overline{\partial_k\nabla\varphi}{\rm d}x-2\int_{B_R}\eta\overline{\partial_k\varphi}\partial_k\mu\nabla\eta\cdot\nabla\varphi{\rm d}x\\
&+2\int_{B_R}\eta\overline{\partial_k\varphi}\partial_k(\mu\nabla\wedge A)\cdot\nabla\eta{\rm d}x+\int_{B_R}\eta^2\partial_k(\mu\nabla\wedge A)\cdot\overline{\partial_k\nabla\varphi}{\rm d}x.
\end{align*}
Let us set that
\begin{align*}
I_1 = &-\int_{B_R}2\eta\overline{\partial_k\varphi}\mu\nabla\eta\cdot\partial_k\nabla\varphi{\rm d}x,\quad I_2=-\int_{B_R}\eta^2\partial_k\mu\nabla\varphi\cdot\overline{\partial_k\nabla\varphi}{\rm d}x, \\
I_3= & -2\int_{B_R}\eta\overline{\partial_k\varphi}\partial_k\mu\nabla\eta\cdot\nabla\varphi{\rm d}x, \quad I_4=2\int_{B_R}\eta\overline{\partial_k\varphi}\partial_k(\mu\nabla\wedge A)\cdot\nabla\eta{\rm d}x,\\
I_5= & \int_{B_R}\eta^2\partial_k(\mu\nabla\wedge A)\cdot\overline{\partial_k\nabla\varphi}{\rm d}x,
\end{align*}
then we can get that
\begin{align*}
I_1\leq&\epsilon\int_{B_R}\eta^2|\partial_k\nabla\varphi|^2{\rm d}x+
C_{\epsilon} ||\mu||_{\mathcal{L}^{\infty}(B_R)}^2\int_{\Omega}|\nabla\eta|^2|\partial_k\varphi|^2{\rm d}x,\\
I_2 \leq & C(\int_{B_R}\eta^2|\partial_k\nabla\varphi|^2{\rm d}x)^{\frac{1}{2}}(\int_{B_R}|\nabla\varphi|^6{\rm d}x)^{\frac{1}{6}}(\int_{B_R}\eta^3|\partial_k\mu|^3{\rm d}x)^{\frac{1}{3}},\\
I_3 \leq & C(\int_{B_R}\eta^3|\nabla\eta|^3|\partial_k\mu|^3{\rm d}x)^{\frac{1}{3}}(\int_{B_R}|\nabla\varphi|^3{\rm d}x)^{\frac{2}{3}},\\
I_4 \leq & \epsilon\int_{B_R}\eta^2|\partial_k(\mu\nabla\wedge A)|^2{\rm d}x+C_{\epsilon} \int_{B_R}|\nabla\eta|^2|\partial_k\varphi|^2{\rm d}x,\\
I_5 \leq &\epsilon\int_{B_R}\eta^2|\partial_k\nabla\varphi|^2{\rm d}x+C_{\epsilon} \int_{B_R}\eta^2|\partial_k(\mu\nabla\wedge A)|^2{\rm d}x,
\end{align*}
where $\epsilon$ is a small constant. Notice that
$$
\begin{aligned}
||\partial_k(\mu\nabla\wedge A)||_{\mathcal{L}^2(B_R)}\leq & ||\mu||_{\mathcal{L}^{\infty}(B_R)}||\partial_k(\nabla\wedge A)||_{\mathcal{L}^2(B_R)}+||\mu||_{\mathcal{W}^{1,3}(B_R)}||\nabla\wedge A||_{\mathcal{L}^6(B_R)}\\
\leq & C (||\mu||_{\mathcal{W}^{1,3}(B_R)}+||\mu||_{\mathcal{L}^{\infty}(B_R)})||\nabla\wedge A||_{\mathcal{W}^{1,2}(B_R)}.
\end{aligned}
$$
Then when $\epsilon$ is small enough, we have $I_1 \leq C||\nabla\varphi||^2_{\mathcal{L}^2(\Omega)}$ and then by applying the embedding theorem we can get that
\begin{align*}
I_2 & \leq C\sum_{k=1}(\int_{B_R}\eta^2|\partial_k\nabla\varphi|^2{\rm d}x)^{\frac{1}{2}}(\int_{B_R}|\nabla^2\varphi|^2{\rm d}x)^{\frac{1}{2}}(\int_{B_R}|\partial_k\mu|^{3+\delta}{\rm d}x)^{\frac{1}{3+\delta}}R^{\frac{\delta}{3+\delta}},\\
I_3 & \leq C\sum_{k}(\int_{B_R}|\partial_k\mu|^3{\rm d}x)^{\frac{1}{3}}(\epsilon(\int_{B_R}|\nabla^2\varphi|^2{\rm d}x)^{\frac{1}{2}}+C(\int_{B_R}|\nabla\varphi|^2{\rm d}x)^{\frac{1}{2}}),\\
I_4+I_5 & \leq C(||\mu||^2_{\mathcal{W}^{1,3}(B_R)}+||\mu||^2_{\mathcal{L}^{\infty}(B_R)})||\nabla\wedge A||^2_{\mathcal{W}^{1,2}(B_R)}+C||\nabla\varphi||^2_{\mathcal{L}^2(B_R)}.
\end{align*}
Since $\sum_{k}\int_{B_R}\eta^2|\partial_k\nabla\varphi|^2{\rm d}x\leq \sum_{i=1}^5 I_i$, and it leads to
\begin{align*}
 \sum_{k}\int_{B_R} & \eta^2|\partial_k\nabla\varphi|^2{\rm d}x\\
 & \leq C\sum_{k=1}(\int_{B_R}\eta^2|\partial_k\nabla\varphi|^2{\rm d}x)^{\frac{1}{2}}(\int_{B_R}|\nabla^2\varphi|^2{\rm d}x)^{\frac{1}{2}}(\int_{B_R}|\partial_k\mu|^{3+\delta}{\rm d}x)^{\frac{1}{3+\delta}}R^{\frac{\delta}{3+\delta}}\\
 + & C\sum_{k}(\int_{B_R}|\partial_k\mu|^3{\rm d}x)^{\frac{1}{3}}(\epsilon(\int_{B_R}|\nabla^2\varphi|^2{\rm d}x)^{\frac{1}{2}}+C(\int_{B_R}|\nabla\varphi|^2{\rm d}x)^{\frac{1}{2}})\\
 +&  C(||\mu||^2_{\mathcal{W}^{1,3}(B_R)}+||\mu||^2_{\mathcal{L}^{\infty}(B_R)})||\nabla\wedge A||^2_{\mathcal{W}^{1,2}(B_R)}+C||\nabla\varphi||^2_{\mathcal{L}^2(B_R)}.
\end{align*}
Again let $\epsilon$ and $R$ be small enough, by using the embedding theorem $||\mu||_{\mathcal{W}^{1,3}(B_R)}+||\mu||_{\mathcal{L}^{\infty}(B_R)}\leq C||\mu||_{\mathcal{W}^{1,3+\delta}(B_R)}$ for any $\delta>0$,
\begin{equation}\label{gradphiw1}
(\int_{B_R}|\nabla^2\varphi|^2{\rm d}x)^{\frac{1}{2}}\leq C(||\nabla\wedge A||_{\mathcal{W}^{1,2}(B_R)}+||\nabla\varphi||_{\mathcal{L}^2(B_R)})\leq C(||\nabla\wedge H||_{\mathcal{L}^2(B_R)}+||H||_{\mathcal{L}^2(B_R)}),
\end{equation}
where the constant $C$ depends on the norm $||\mu||_{\mathcal{W}^{1,3+\delta}}$.
Combining the equations (\ref{gradphiw1}) with (\ref{curlH}),(\ref{curlA})and (\ref{gradphi}), then we obtain the $\mathcal{H}^1$ interior estimate of $H$ as claimed in the proposition.

\medskip

{\bf Step 4}. \emph{Boundary estimates}.
Notice that $\varphi$ is a solution of the conormal derivatives problem of a scalar elliptic equation of second order, so one can use the standard argument to derive the boundary estimates of the solutions. More precisely, for any point $x_0\in\partial\Omega$, we can introduce an orthogonal transformation of the coordinates,
$$
y=\Phi(x),\qquad x=\Psi(y).
$$
In the new coordinates $y=(y_1,y_2,y_3)$, and we define $B_{R+}(y_0):=B_R(y_0)\cap\{y_3>0\}$, where $R>0$
%$\Omega$ in $B_R(y_0)$
and $y_0=\Phi(x_0)$. Let $\tilde{\varphi}(y)=\varphi(\Psi(y))$, and for the simplicity, we write $\varphi$ instead of $\tilde{\varphi}$. Then  $\varphi$ satisfies the equation
\begin{equation}\label{eqcono}
{\rm div}(\tilde{\mu}\nabla\varphi)-\tilde{b}\cdot\nabla\varphi=-{\rm div}(\bar{\mu}\nabla_{x}\wedge A)+\bar{b}\cdot\nabla_x\wedge A,
\end{equation}
where $\nabla_x=(\partial_{x_1},\partial_{x_2},\partial_{x_3})$, $\tilde{\mu}_{kl}=\sum_{i,j=1}^3\frac{\partial y_k}{\partial x_i}\frac{\partial y_l}{\partial x_j}\mu_{ij}$, $\bar{\mu}_{kj}=\sum_{i=1}^3\frac{\partial y_k}{\partial x_i}\mu_{ij}$, $\tilde{b}_m=\sum_{i,j,k,l=1}^3\frac{\partial^2y_k}{\partial x_i\partial x_l}\frac{\partial y_m}{\partial x_j}\frac{\partial y_k}{\partial x_l}\mu_{ij}$, and $\bar{b}_j=\sum_{i,k,l=1}^3\frac{\partial^2y_k}{\partial x_i\partial x_l}\frac{\partial y_k}{\partial x_l}\mu_{ij}$.
Moreover, $\varphi$ satisfies the following boundary condition on $\{y_3=0\}$
\begin{equation}\label{bdycond}
(\tilde{\mu}\nabla\varphi)\cdot e_3=-(\bar{\mu}\nabla_x\wedge A)\cdot e_3,
\end{equation}
where $e_3$ is the unit direction of $y_3$-axis.
Since the boundary $\partial\Omega$ is of $ C^{1,1}$,  the coefficient $\tilde{\mu}$ and $\bar{\mu}$ have the same regularity as $\mu$, while $\tilde{b}$ and $\bar{b}$ are $\mathcal{L}^{\infty}$ functions.

Again let $\eta$ be a smooth positive cut-off function such that $\eta\equiv1$ in $B_r(y_0)$ and $\eta=0$ on $\partial B_R(y_0)$. Then let us take the tangential derivatives $\partial_{k}$, where $\partial_k=\partial_{y_k}$ with $k=1$ or $2$ on both sides of the equation \eqref{eqcono} and multiply $\eta^2\partial_k\overline{\varphi}$. After an integration by part, it gives that
\begin{align*}
&-\int_{B_{R+}(y_0)}\partial_k(\tilde{\mu}\nabla\varphi)\cdot \overline{\nabla(\eta^2\partial_k\varphi)}{\rm d}y+\int_{B_{R+}(y_0)}\tilde{b}\cdot\nabla\varphi\overline{\partial_k(\eta^2\partial_k\varphi)}{\rm d}y\\
&+\int_{(y_3=0)\cap B_{R+}(y_0)}\eta^2\overline{\partial_k\varphi}\partial_k(\tilde{\mu}\nabla\varphi)\cdot e_3{\rm d}\sigma\\
        =&\int_{B_{R+}(y_0)}\partial_k(\bar{\mu}\nabla\wedge A)\cdot \overline{\nabla(\eta^2\partial_k\varphi)}{\rm d}y-\int_{B_{R+}(y_0)}\bar{b}\cdot\nabla_x\wedge A\overline{\partial_k(\eta^2\partial_k\varphi)}{\rm d}y\\
&-\int_{(y_3=0)\cap B_{R+}(y_0)}\eta^2\overline{\partial_k\varphi}\partial_k(\bar{\mu}\nabla\wedge A)\cdot e_3{\rm d}\sigma .
\end{align*}
Applying the boundary condition \eqref{bdycond}, we have that
$$
\partial_k(\tilde{\mu}\nabla\varphi)\cdot e_3=-\partial_k(\bar{\mu}\nabla\wedge A)\cdot e_3,\qquad\mbox{on } y_3=0,
$$
for $k=1$ or $2$.
Therefore, the boundary terms vanish and we have
\begin{align*}
&-\int_{B_{R+}(y_0)}\partial_k(\tilde{\mu}\nabla\varphi)\cdot \overline{\nabla(\eta^2\partial_k\varphi)}{\rm d}y+\int_{B_{R+}(y_0)}\tilde{b}\cdot\nabla\varphi\overline{\partial_k(\eta^2\partial_k\varphi)}{\rm d}y\\
%&+\int_{(y_3=0)\cap B_{R+}(y_0)}\eta^2\overline{\partial_k\varphi}\partial_k(\tilde{\mu}\nabla\varphi)\cdot e_3{\rm d}\sigma\\
        =&\int_{B_{R+}(y_0)}\partial_k(\bar{\mu}\nabla\wedge A)\cdot \overline{\nabla(\eta^2\partial_k\varphi)}{\rm d}y-\int_{B_{R+}(y_0)}\bar{b}\cdot\nabla_x\wedge A\overline{\partial_k(\eta^2\partial_k\varphi)}{\rm d}y.
%&-\int_{(y_3=0)\cap B_{R+}(y_0)}\eta^2\overline{\partial_k\varphi}\partial_k(\bar{\mu}\nabla\wedge A)\cdot e_3{\rm d}\sigma .
\end{align*}
Then following {\bf Step 3}, we can derive the estimates of $\partial_k\nabla\varphi$ for $k=1$ or $2$.

By applying the equation \eqref{equ:varphi3.1}, it is easy to see that $\partial_{y_3}^2\varphi$ is bounded by $\partial_{k}\nabla\varphi$ for $k=1$ or $2$. Then  we have the boundary estimates in the $y$-coordinates which satisfies that

\begin{align*}
\int_{B_{r+}(y_0)}|\nabla^2\varphi|^2{\rm d}x & \leq C(||\nabla\wedge A||_{\mathcal{W}^{1,2}({ B_{R+}(y_0)})}+||\nabla\varphi||_{\mathcal{L}^2({ B_{R+}(y_0)})}) \\ & \leq C(||\nabla\wedge H||_{\mathcal{L}^2(B_{R+}(y_0))}+||H||_{\mathcal{L}^2( B_{R+}(y_0))}).
\end{align*}

\medskip
{\bf Step 5}. \emph{Global estimates}.
Finally, let $\eta_i$ be cut-off functions which satisfies $\sum_i \eta_i\equiv 1$ and such that the set of all the subregions 
$$
\Omega_i:=\{x\in\Omega\,;\,\eta_i(x)>0\}
$$ 
together is a finite cover of $\Omega$ with the property that $\rm{diam}\{\Omega_i\}\leq R$.
Based on the interior and boundary estimate, we can prove the $\mathcal{H}^1$-regularity of $H$ as follows
\begin{align*}
||H||_{\mathcal{H}^{1}(\Omega)}
\leq & ||\nabla(\nabla\varphi)||_{\mathcal{L}^2(\Omega)}+||\nabla(\nabla\wedge A)||_{\mathcal{L}^2(\Omega)}+\|H\|_{\mathcal{L}^2(\Omega)}\\
\leq & C( \sum_{i}||\eta_i\nabla(\nabla\varphi)||_{\mathcal{L}^2(\Omega_i)}+||\nabla\wedge H||_{\mathcal{L}^2(\Omega)}+||H||_{\mathcal{L}^2(\Omega)})\\
\leq &C(||\nabla\wedge H||_{\mathcal{L}^2(\Omega)}+||H||_{\mathcal{L}^2(\Omega)})\\
\leq &C||H||_{\mathcal{L}^2(\Omega)}.
\end{align*}
\end{proof}

Based on the new approach, we are now going to prove theorem \ref{thm:jejm} for the non-homogeneous case.
%In this section, we assume that the Maxwell's equation has source terms, and then we show similar regularity results of the solutions as given in the homogeneous case. As before, we first show the $\mathcal{H}^1$ estimate as addressed in theorem \ref{thm:jejm}. %Now we prove the theorem as follows.
\begin{proof}[Proof of theorem \ref{thm:jejm}]
Let
\begin{equation}
\tilde{E}=E-G,\qquad\tilde{H}=H.
\end{equation}
Obviously, $\tilde{E}$ and $\tilde{H}$ satisfy the system
\begin{equation}
\begin{array}{cc}
{\begin{array}{c}
{\nabla\wedge{\tilde{E}}  -i\,\omega\mu_{0}\mu(x)\,{\tilde{H}}=\tilde{J}_m}\\
{\nabla\wedge{\tilde{H}}  +i\,\omega\varepsilon_{0}\varepsilon(x)\,{\tilde{E}}=\tilde{J}_e} \end{array} }&{\textrm{ in }\Omega,}\\
{\tilde{E}\wedge\boldsymbol{\nu}  = 0}&{\textrm{ on }\partial\Omega.}
\end{array}
\end{equation}
where
\begin{equation*}
\tilde{J}_m:= J_m-\nabla\wedge G , \quad  \tilde{J}_e:= J_e-i\,\omega\varepsilon_{0}\varepsilon(x)\,{G} \quad \textrm{ in }\Omega.
\end{equation*}
Then $\tilde{H}$ satisfies the equations
\begin{equation}\label{equ:4.3a}
\nabla\wedge(\varepsilon^{-1}(x)\nabla\wedge \tilde{H})=k_1^2\mu(x)\tilde{H}-k_2\tilde{J}_m+\nabla\wedge(\varepsilon^{-1}(x)\tilde{J}_e),
\end{equation}
and
\begin{equation}\label{equ:4.4a}
{\rm div}(\mu(x)\tilde{H})=k_3{\rm div} \tilde{J}_m=k_3{\rm div} J_m.
\end{equation}
%where the constants $k_1^2=\omega^2\varepsilon_0\mu_0$, $k_2=i\omega\varepsilon_0$, and $k_3=i/(\omega\mu_0)$.
On the boundary $\partial\Omega$, $\tilde{H}$ satisfies that
\begin{equation}\label{identity1}
(\varepsilon^{-1}\nabla\wedge\tilde{H})\wedge\boldsymbol{\nu}=-iw\varepsilon_0\tilde{E}\wedge\boldsymbol{\nu}+\varepsilon^{-1}\tilde{J}_e\wedge\boldsymbol{\nu}=\varepsilon^{-1}\tilde{J}_e\wedge\boldsymbol{\nu}.
\end{equation}

{\bf Step 1}. \emph{Interior Estimates}

Multiply $\overline{\tilde{H}}$ on both sides of equation \eqref{equ:4.3a} and integrate it by part, then
\begin{equation}\label{interorHtild}
\int_{\Omega}\varepsilon^{-1}(x)(\nabla\wedge \tilde{H})\cdot\overline{(\nabla\wedge \tilde{H})}dx+\int_{\partial\Omega}((\varepsilon^{-1}\nabla\wedge \tilde{H})\wedge\boldsymbol{\nu})\cdot \overline{\tilde{H}}{\rm d}\sigma\\
\end{equation}
\begin{equation*}
=k_1^2\int_{\Omega}(\mu\tilde{H})\cdot\overline{\tilde{H}}dx-k_2\int_{\Omega}\tilde{J}_m\cdot\overline{\tilde{H}}{\rm d}x+\int_{\Omega}(\varepsilon^{-1}\tilde{J}_e)\cdot\overline{(\nabla\wedge \tilde{H})}{\rm d}x+\int_{\partial\Omega}(\varepsilon^{-1}\tilde{J}_e\wedge\boldsymbol{\nu})\cdot\overline{\tilde{H}}{\rm d}\sigma.
\end{equation*}
By applying the identity \eqref{identity1} to the equation \eqref{interorHtild}, it is easy to see that the following inequality holds,

\begin{equation}\label{3.11xwx}
||\nabla\wedge \tilde{H}||_{\mathcal{L}^2(\Omega)}\le C(||\tilde{H}||_{\mathcal{L}^2(\Omega)}+||\tilde{J}_e||_{\mathcal{L}^2(\Omega)}+||\tilde{J}_m||_{\mathcal{L}^2(\Omega)}).
\end{equation}

Next, let $\tilde{H}=\nabla\varphi+\nabla\wedge A$, and we have that %Similar to the argument in Step 2 of the proof of Proposition \ref{thm:1}, we have that
$$
||\nabla\wedge A||_{\mathcal{H}^1(\Omega)}\le \|\nabla\wedge\tilde{H}\|_{\mathcal{L}^2(\Omega)}\le C(||\tilde{H}||_{\mathcal{L}^2(\Omega)}+||\tilde{J}_e||_{\mathcal{L}^2(\Omega)}+||\tilde{J}_m||_{\mathcal{L}^2(\Omega)}),
$$
and $\varphi$ satisfies the following scalar elliptic equation of second order
\begin{equation}\label{inhomodiveq1}
{\rm div}(\mu\nabla\varphi)=-{\rm div}(\mu\nabla\wedge A)+k_3{\rm div}J_m.
\end{equation}
%with the boundary condition,
%$$
%(\mu\nabla\varphi)\cdot\boldsymbol{\nu}
%$$

Then take derivatives $\partial_{k}$ on both sides of the equation \eqref{inhomodiveq1}and multiply $\eta^2\partial_k\overline{\varphi}$. After an integration by part, it gives that
\begin{align*}
\lambda_1\int_{\Omega}& \eta^2|\partial_k\nabla\varphi|^2{\rm d}x\leq \int_{\Omega}\eta^2\mu\partial_k\nabla\varphi\cdot\overline{\partial_k\nabla\varphi}{\rm d}x\\
=&-\int_{\Omega}2\eta\partial_k\varphi\mu\nabla\eta\cdot\overline{\partial_k\nabla\varphi}{\rm d}x-\int_{\Omega}\eta^2\partial_k\mu\nabla\varphi\cdot\overline{\partial_k\nabla\varphi}{\rm d}x-2\int_{\Omega}\eta\partial_k\varphi\partial_k\mu\nabla\eta\cdot\overline{\nabla\varphi}{\rm d}x\\
&+2\int_{\Omega}\eta\overline{\partial_k\varphi}\partial_k(\mu\nabla\wedge A)\cdot\nabla\eta{\rm d}x+\int_{\Omega}\eta^2\partial_k(\mu\nabla\wedge A)\cdot\overline{\partial_k\nabla\varphi}{\rm d}x \\
&+2\int_{\Omega}\eta\overline{\partial_k\varphi}\partial_k(k_3J_m)\cdot\nabla\eta{\rm d}x+\int_{\Omega}\eta^2\partial_k(k_3J_m)\cdot\overline{\partial_k\nabla\varphi}{\rm d}x.
\end{align*}
%%%%%%%%%%after this line

Similar to the approach that we used in the argument for the homogeneous equations, we use the following notation,
\begin{align*}
\mathscr{I}_1 = &-\int_{B_R}2\eta\overline{\partial_k\varphi}\mu\nabla\eta\cdot\partial_k\nabla\varphi{\rm d}x,\quad \mathscr{I}_2=-\int_{B_R}\eta^2\partial_k\mu\nabla\varphi\cdot\overline{\partial_k\nabla\varphi}{\rm d}x, \\
\mathscr{I}_3= & -2\int_{B_R}\eta\overline{\partial_k\varphi}\partial_k\mu\nabla\eta\cdot\nabla\varphi{\rm d}x, \quad \mathscr{I}_4=2\int_{B_R}\eta\overline{\partial_k\varphi}\partial_k(\mu\nabla\wedge A)\cdot\nabla\eta{\rm d}x,\\
\mathscr{I}_5= & \int_{B_R}\eta^2\partial_k(\mu\nabla\wedge A)\cdot\overline{\partial_k\nabla\varphi}{\rm d}x,\quad \mathscr{I}_6=  2\int_{\Omega}\eta\overline{\partial_k\varphi}\partial_k(k_3J_m)\cdot\nabla\eta{\rm d}x, \\
\mathscr{I}_7= & \int_{\Omega}\eta^2\partial_k(k_3J_m)\cdot\overline{\partial_k\nabla\varphi}{\rm d}x,
\end{align*}
then we can get that
\begin{align*}
\mathscr{I}_1\leq&\epsilon\int_{B_R}\eta^2|\partial_k\nabla\varphi|^2{\rm d}x+
C_{\epsilon} ||\mu||_{\mathcal{L}^{\infty}(B_R)}^2\int_{\Omega}|\nabla\eta|^2|\partial_k\varphi|^2{\rm d}x,\\
\mathscr{I}_2 \leq & C(\int_{B_R}\eta^2|\partial_k\nabla\varphi|^2{\rm d}x)^{\frac{1}{2}}(\int_{B_R}|\nabla\varphi|^6{\rm d}x)^{\frac{1}{6}}(\int_{B_R}\eta^3|\partial_k\mu|^3{\rm d}x)^{\frac{1}{3}},\\
\mathscr{I}_3 \leq & C(\int_{B_R}\eta^3|\nabla\eta|^3|\partial_k\mu|^3{\rm d}x)^{\frac{1}{3}}(\int_{B_R}|\nabla\varphi|^3{\rm d}x)^{\frac{2}{3}},\\
\mathscr{I}_4 \leq & \epsilon\int_{B_R}\eta^2|\partial_k(\mu\nabla\wedge A)|^2{\rm d}x+C_{\epsilon} \int_{B_R}|\nabla\eta|^2|\partial_k\varphi|^2{\rm d}x,\\
\mathscr{I}_5 \leq &\epsilon\int_{B_R}\eta^2|\partial_k\nabla\varphi|^2{\rm d}x+C_{\epsilon} \int_{B_R}\eta^2|\partial_k(\mu\nabla\wedge A)|^2{\rm d}x,\\
\mathscr{I}_6 \leq & \epsilon \int_{B_R}\eta^2|\partial_k(J_m)|^2{\rm d}x+C_{\epsilon} \int_{B_R}|\nabla\eta|^2|\partial_k\varphi|^2{\rm d}x, \\
\mathscr{I}_7 \leq &  \epsilon \int_{B_R}\eta^2|\partial_k\nabla \varphi|^2{\rm d}x+C_{\epsilon} \int_{B_R}\eta^2|\partial_k(J_m)|^2{\rm d}x.
\end{align*}
Therefore, when $\epsilon$ in above inequalities is small enough, we have the interior estimates for any $x_0\in\Omega$,
$$
\begin{aligned}
&\big(\int_{B_{r}(x_0)}|\nabla^2\varphi|^2{\rm d}x\big)^{\frac{1}{2}}\\
\leq & C(||\nabla\wedge A||_{\mathcal{W}^{1,2}(B_R)}+||J_m||_{\mathcal{W}^{1,2}(B_R)}+||\nabla\varphi||_{\mathcal{L}^2(B_R)})\\
\leq & C(||\nabla\wedge \tilde{H}||_{\mathcal{L}^2(B_R)}+||J_m||_{\mathcal{W}^{1,2}(B_R)}+||\tilde{J}_m||_{\mathcal{L}^{2}(B_R)}+||\tilde{J}_e||_{\mathcal{L}^{2}(B_R)}+||\tilde{H}||_{\mathcal{L}^2(B_R)}).
\end{aligned}
$$
{\bf Step 2}. \emph{Boundary Estimates.}
For the boundary estimates, by \eqref{3.11xwx}, it is easy to see that for any $x_0\in\partial\Omega$,

\begin{equation}
\int_{B_{r+}\cap\Omega}|\nabla\wedge \tilde{H}|^2dx \le C(R)\int_{\Omega}(|\tilde{H}|^2+|\tilde{J}_e|^2+|\tilde{J}_m|^2)dx.
\end{equation}

Moreover, in the new $y$-coordinates as introduced in {\bf Step 4} of the proof of Proposition \ref{thm:1}, $\varphi := \tilde{\varphi}(y)$  satisfies the equation
\begin{equation}\label{diveqbound1}
{\rm div}(\tilde{\mu}\nabla\varphi)-\tilde{b}\cdot\nabla\varphi=-{\rm div}(\bar{\mu}\nabla_{x}\wedge A)+\bar{b}\cdot\nabla_x\wedge A+k_3{\rm div}_xJ_m,
\end{equation}
with the boundary condition on $\{y_3=0\}$
\begin{equation}\label{3.19a}
(\tilde{\mu}\nabla\varphi)\cdot e_3=-(\bar{\mu}\nabla_x\wedge A)\cdot e_3+k_3\tilde{J}_m\cdot e_3.
\end{equation}
Let $\eta$ be a smooth positive cut-off function such that $\eta\equiv1$ in $B_r(y_0)$ and $\eta=0$ on $\partial B_R(y_0)$. Then take the tangential derivatives $\partial_{k}$ on equation (\ref{diveqbound1}), where $\partial_{k}=\partial_{y_k}$ ($k=1$ or $2$) and multiply $\eta^2\partial_k\overline{\varphi}$. After an integration by part, it gives that
\begin{align*}
&-\int_{B_{R+}(y_0)}\partial_k(\tilde{\mu}\nabla\varphi)\cdot \overline{\nabla(\eta^2\partial_k\varphi)}{\rm d}y+\int_{B_{R+}(y_0)}\tilde{b}\cdot\nabla\varphi\overline{\partial_k(\eta^2\partial_k\varphi)}{\rm d}y\\
&+\int_{\{y_3=0\}\cap B_{R+}(y_0)}\eta^2\overline{\partial_k\varphi}\partial_k(\tilde{\mu}\nabla\varphi)\cdot e_3{\rm d}\sigma\\
=&\int_{B_{R+}(y_0)}\partial_k(\tilde{\mu}\nabla_x\wedge A)\cdot \overline{\nabla(\eta^2\partial_k\varphi)}{\rm d}y-\int_{B_{R+}(y_0)}\bar{b}\cdot\nabla_x\wedge A\overline{\partial_k(\eta^2\partial_k\varphi)}{\rm d}y\\
&-\int_{B_{R+}(y_0)}k_3{\rm div}_{x}J_m\overline{\partial_k(\eta^2\partial_k\varphi)}{\rm d}y-\int_{\{y_3=0\}\cap B_{R+}(y_0)}\eta^2\overline{\partial_k\varphi}\partial_k(\tilde{\mu}\nabla_x\wedge A)\cdot e_3{\rm d}\sigma .
\end{align*}
Here we use the fact that all the integrals on the boundary $\partial B_{R+}(y_0)\cap\{y_0>0\}$ is $0$, due to the cut-off function $\eta$.

Then by the boundary condition \eqref{3.19a} and the trace theorem,
\begin{align*}
&\int_{\{y_3=0\}\cap B_{R+}(y_0)}\eta^2\overline{\partial_k\varphi}\partial_k(\tilde{\mu}\nabla\varphi)\cdot e_3{\rm d}\sigma+\int_{\{y_3=0\}\cap B_{R+}(y_0)}\eta^2\overline{\partial_k\varphi}\partial_k(\tilde{\mu}\nabla_x\wedge A)\cdot e_3{\rm d}\sigma\\
= &k_3\int_{\{y_3=0\}\cap B_{R+}(y_0)}\eta^2\overline{\partial_k\varphi}\partial_k\tilde{J}_m\cdot e_3{\rm d}\sigma\\
\leq& C||\eta^2\partial_k\varphi||_{\mathcal{H}^{\frac{1}{2}}(\{y_3=0\}\cap B_{R+}(y_0))}||\partial_k(\tilde{J}_m\cdot\boldsymbol{\nu})||_{\mathcal{H}^{-\frac{1}{2}}(\{y_3=0\}\cap B_{R+}(y_0))}\\
\leq& C||\varphi||_{\mathcal{H}^{2}(B_{R+}(y_0))}||\tilde{J}_m\cdot\boldsymbol{\nu}||_{\mathcal{H}^{\frac{1}{2}}(\{y_3=0\}\cap B_{R+}(y_0))}.
\end{align*}
Therefore, by the H\"{o}lder inequality, if the radius $R$ is small enough, then
\begin{align*}
( &\int_{B_{r+}(y_0)}|\nabla^2\varphi|^2{\rm d}x)^{\frac{1}{2}}\\
  \leq& C(||\nabla\wedge A||_{\mathcal{W}^{1,2}({B_{R+}(y_0)})}+||\nabla\varphi||_{\mathcal{L}^2({ B_{R+}(y_0)})}+||J_m||_{\mathcal{W}^{1,2}(B_{R+})}+||\tilde{J}_m\cdot\boldsymbol{\nu}||_{\mathcal{H}^{\frac{1}{2}}(\{y_3=0\}\cap B_{R+}(y_0))}) \\
\leq &C(||\tilde{H}||_{\mathcal{L}^2(B_{R+})}+||\tilde{J}_e||_{\mathcal{L}^2(B_{R+})}+||\tilde{J}_m||_{\mathcal{L}^2(B_{R+})}+||J_m||_{\mathcal{W}^{1,2}(B_{R+})}+||\tilde{J}_m\cdot\boldsymbol{\nu}||_{\mathcal{H}^{\frac{1}{2}}(\{y_3=0\}\cap B_{R+}(y_0))})\\
\leq& C(||\nabla\wedge G||_{\mathcal{L}^2(B_{R+})}+||J_m||_{\mathcal{W}^{1,2}(B_{R+})}\\
 &~~~+||H||_{\mathcal{L}^2(B_{R+})}+||J_e||_{\mathcal{L}^2(B_{R+})}+||(J_m-\nabla\wedge G)\cdot\boldsymbol{\nu}||_{\mathcal{H}^{\frac{1}{2}}(\{y_3=0\}\cap B_{R+}(y_0))}).
 \end{align*}

Hence by summarizing all the interior and boundary estimates together, we have that
\begin{align*}
||H||_{\mathcal{H}^{1}(\Omega)}=&||\tilde{H}||_{\mathcal{H}^{1}(\Omega)}\\
\leq& ||\nabla\wedge A||_{\mathcal{L}^2(\Omega)}+||\varphi||_{\mathcal{H}^2(\Omega)}\\
\leq &C(||\tilde{H}||_{\mathcal{L}^2(\Omega)}+||\tilde{J}_e||_{\mathcal{L}^2(\Omega)}+||\tilde{J}_m||_{\mathcal{L}^2(\Omega)}+||J_m||_{\mathcal{W}^{1,2}(\Omega)}+||\tilde{J}_m\cdot\boldsymbol{\nu}||_{\mathcal{H}^{\frac{1}{2}}(\partial\Omega)})\\
\leq& C(||\nabla\wedge G||_{\mathcal{L}^2(\Omega)}+||J_m||_{\mathcal{W}^{1,2}(\Omega)}\\
 &~~~+||H||_{\mathcal{L}^2(\Omega)}+||J_e||_{\mathcal{L}^2(\Omega)}+||(J_m-\nabla\wedge G)\cdot\boldsymbol{\nu}||_{\mathcal{H}^{\frac{1}{2}}(\partial\Omega)}).
\end{align*}

It completes the proof of this theorem.
\end{proof}
                       %%%%%%%%%%%%%%%%%%%%%%%%%%%%%%%%%%%%%%%%%%%%%%%%%%%%%%%%%%%%%%%%%%%%%%%%%%%%%%%%%%%%%%%%%%%%%%%%%%%%%%%%%%%%%%%%%%%%%%%%%%%%%%

                       %%%%%%%%%%%%%%%%%%%%%%%%%%%%%%%%%%%%%%%%%%%%%%%%%%   Proof of Theorem 2   %%%%%%%%%%%%%%%%%%%%%%%%%%%%%%%%%%%%%%%%%%%%%%%%%%%%

                       %%%%%%%%%%%%%%%%%%%%%%%%%%%%%%%%%%%%%%%%%%%%%%%%%%%%%%%%%%%%%%%%%%%%%%%%%%%%%%%%%%%%%%%%%%%%%%%%%%%%%%%%%%%%%%%%%%%%%%%%%%%%%%

\subsection{Proof of Theorem \ref{thm:inw1p}}
As before, we develop the exact regularity results for the homogeneous equation before giving the proof of the theorem (\ref{thm:inw1p}).
Notice that in the bounded domain $\Omega$, the Rellich-Kondrachov Theorem provides that , $||H||_{\mathcal{L}^{p}(\Omega)}\leq\epsilon||\nabla H||_{\mathcal{L}^p(\Omega)}+C||H||_{\mathcal{L}^2(\Omega)}$ for $p \geq 2$, so we only need to estimate $||\nabla H||_{\mathcal{L}^p(\Omega)}$. We first give the following Lemma to estimate the solutions of elliptic equations with constant coefficients. For any $x_0\in\Omega$, let $0<R\leq\rm{dist}(x_0,\,\partial\Omega)$, and let $B_R=B_R(x_0):=\{x\,;\,|x-x_0|< R\}$.

\begin{lem}\label{lem:3}
If $f^{m}_{i} \in \mathcal{L}^{p}(B_R)$, suppose $u \in \mathcal{H}_0^1(B_R)$ satisfies the following equation
$$
\sum_{m,\beta,i,j=1}^3\int_{B_R}A_{ij}^{m\beta} \nabla_{m}u^i \nabla_{\beta}\varphi^jdx=\sum_{i,m=1}^3\int_{B_R}f^{m}_i\nabla_{m}\varphi^i dx, \quad \forall \varphi\in H_0^1(B_R),$$
where $\nabla=(\nabla_1,\nabla_2,\nabla_3)$, $A_{ij}^{m\beta}$ are constants and satisfy that $|A_{ij}^{m\beta}| \le M $.Further there is a positive constant $\lambda>0$ such that $\sum_{i,j,m,\beta=1}^3A_{ij}^{m\beta}\xi^i_{m}\xi^j_{\beta}\ge \lambda |\xi|^2$ for any $\xi_{m}^i,\xi_{\beta}^j\in\mathbb{R}$, then for any $2\leq p<\infty$,  $$ ||\nabla u||_{\mathcal{L}^p(B_R)}\le C||f||_{\mathcal{L}^p(B_R)}.$$
where $C$ depends on $p$.\end{lem}
\begin{proof}
By choosing $\varphi=\overline{u}$, it is easy to see that $||\nabla u||_{\mathcal{L}^2(B_R)}\le C||f||_{\mathcal{L}^2(B_R)}$. Again it is not hard to obtain the estimate $||\nabla u||_{\operatorname{BMO}(B_R)}\le||\nabla u||_{\mathfrak{L}^{2,n}(B_R)}\le C||f||_{\mathfrak{L}^{2,n}(B_R)}\leq C||f||_{\mathcal{L}^{\infty}(B_R)}$ for the elliptic equations with constant coefficients where the $\operatorname{BMO}$ space and $\mathfrak{L}^{2,n}$ space are introduced in the Appendix. Now the Lemma is proved by the Stampacchia interpolation theorem. See the details of the Stampacchia interpolation theorem in the Appendix.
\end{proof}

\begin{prop}\label{thm:2}Let $\mu(x)\in \mathcal{W}^{1,3+\delta}(\Omega)^{3 \times 3}$, where $\delta>0$.
Suppose $H\in \mathcal{H}^1(\Omega)$ is a weak solution of the equations
\eqref{eq1}, with the boundary condition $(\varepsilon^{-1}\nabla\wedge H)\wedge \boldsymbol{\nu}=0$ on $\partial\Omega$,
where $\varepsilon^{-1}(x)\in \mathcal{C}^0(\Omega)^{3\times3}$ and satisfies the condition \eqref{cond1},
then the following inequality holds
\begin{equation}
||H||_{\mathcal{W}^{1,p}(\Omega)}\leq C||H||_{\mathcal{L}^2(\Omega)},\qquad p \geq 2,
\end{equation}
where the constant $C$ does not depend on the solutions $H$.
\end{prop}

\begin{proof}[Proof of Proposition \ref{thm:2}]
%For the simplicity, we first work on the case in which $\mu=Id^{3\times3}$, then we generalize the method to the case $\mu \neq Id^{3\times3}$ using the similar %idea.\\

%$\spadesuit $ \,\,{\bf The case that $\mu=Id^{3\times3}$. }

We divide the proof into three steps.

\smallskip
{\bf Step 1.} \emph{Interior regularity.}
For any $x_0\in\Omega$, let $B_R:=B_R(x_0) \subset \Omega$.
Suppose $\eta$ is a cut-off function so that $0\leq\eta\leq1$ in $B_R$, $\eta=1$ in $B_r$ and $\eta=0$ on $\partial B_{R}$, where $0<r<R$. Moreover, $|\nabla\eta|\leq\frac{C}{R-r}$. %Assume $\mu=Id^{3\times3}$ and then rewrite the equation \eqref{eq1} as follows
%\begin{equation}
%{\nabla\wedge(\varepsilon^{-1}(x_0)\nabla\wedge H)=k^2H+\nabla\wedge([\varepsilon^{-1}(x_0)-\varepsilon^{-1}(x)]\nabla\wedge H)}, \qquad {\textrm{in} \quad \Omega}
%\end{equation}
%multiply the equation by $\varphi(x) \in \mathcal{H}^1_0(\Omega)$ and then after an integration by part, we get that

For any $\psi(x) \in \mathcal{H}^1_0(B_R(x_0))$, we have that

\begin{equation}
\begin{aligned}
&\lefteqn{\int_{B_R} \varepsilon^{-1}(x_0)\nabla\wedge(\eta H)\cdot\nabla\wedge\psi(x) dx{}}\\
=&\int_{B_R} \lbrack \varepsilon^{-1}(x_0)-\varepsilon^{-1}(x)\rbrack \nabla\wedge(\eta H)\cdot\nabla\wedge\psi(x) dx
+\int_{B_R} \varepsilon^{-1}(x)\nabla\wedge(\eta H)\cdot\nabla\wedge\psi(x) dx\\
=&\int_{B_R} \lbrack \varepsilon^{-1}(x_0)-\varepsilon^{-1}(x)\rbrack \nabla\wedge(\eta H)\cdot\nabla\wedge\psi(x) dx+
\int_{B_R} \varepsilon^{-1}(x)\nabla\eta\wedge H\cdot\nabla\wedge\psi(x) dx{}\\
&+\int_{B_R} \eta\varepsilon^{-1}(x)\nabla\wedge H\cdot\nabla\wedge\psi(x) dx\\
=&\int_{B_R} \lbrack \varepsilon^{-1}(x_0)-\varepsilon^{-1}(x)\rbrack \nabla\wedge(\eta H)\cdot\nabla\wedge\psi(x) dx+
\int_{B_R} \varepsilon^{-1}(x)\nabla\wedge H\cdot\nabla\wedge(\eta\psi(x)) dx{}\\
&-\int_{B_R} \varepsilon^{-1}(x)\nabla\wedge H\cdot \nabla\eta\wedge\psi(x) dx+\int_{B_R}\varepsilon^{-1}(x)\nabla\eta\wedge H\cdot\nabla\wedge\psi(x) dx\\
=&\int_{B_R} \lbrack \varepsilon^{-1}(x_0)-\varepsilon^{-1}(x)\rbrack \nabla\wedge(\eta H)\cdot\nabla\wedge\psi(x) dx+\int_{B_R}\mathbb{G}_i\psi^idx+\int_{B_R}\mathbb{F}^{m}_i\nabla_{m}\psi^i dx,
\end{aligned}
\end{equation}
where, by equation \eqref{eq1}, we assume that
\begin{equation*}
\begin{aligned}
&\int_{B_R}\mathbb{G}_i\psi^idx=\int_{B_R}k^2 \mu(x)H\cdot\eta\psi dx-\int_{B_R} \varepsilon^{-1}(x)\nabla\wedge H\cdot \nabla\eta\wedge\psi(x)dx,\\
&\int_{B_R}\mathbb{F}^{m}_i\nabla_{m}\psi^i dx=\int_{B_R}\varepsilon^{-1}(x)\nabla\eta\wedge H\cdot\nabla\wedge\psi(x) dx.
\end{aligned}
\end{equation*}
Let $\omega\in \mathcal{H}^1_0(B_R)$ satisfy
\begin{equation}\label{omega1}
-\int_{B_R}\delta^{m\beta}\delta_{ij}\nabla_{\beta}\omega^j_s\nabla_{m}\phi^idx=\int_{B_R}\delta^{m s}\mathbb{G}_i\nabla_{m}\phi^i d x, \quad \forall \phi \in \mathcal{C}^{1,1}(B_R),
\end{equation}
then we can write that
\begin{align}\label{etaHestimate}
\int_{B_R} \varepsilon^{-1}(x_0)\nabla\wedge(\eta H)\cdot\nabla\wedge\psi(x) dx=&
\int_{B_R} \lbrack \varepsilon^{-1}(x_0)-\varepsilon^{-1}(x)\rbrack \nabla\wedge(\eta H)\cdot\nabla\wedge\psi(x) dx\\+&\int_{B_R}\tilde{\mathbb{F}}^{m}_i\nabla_{m}\psi^i dx, \nonumber
\end{align}
where $\tilde{\mathbb{F}}^{m}_i=\mathbb{F}^{m}_i+\omega^i_{m}$.\\

By applying Lemma \ref{lem:3} to equation \eqref{omega1}, for any $p\geq2$, we have that

\begin{equation}
||\omega||_{\mathcal{L}^{p^*}(B_R)}\leq C||\nabla \omega||_{\mathcal{L}^p(B_R)}\leq||\mathbb{G}||_{\mathcal{L}^p(B_R)},
\end{equation}
where $p^*$ is the Sobolev conjugate of $p$, so $\frac{1}{p^*}=\frac{1}{p}-\frac{1}{n}$ .
By applying the H\"older inequality to equation \eqref{etaHestimate}, we obtain that
\begin{equation}
||\nabla\wedge(\eta H)||_{\mathcal{L}^q(B_R)}\leq C||\lbrack \varepsilon^{-1}(x_0)-\varepsilon^{-1}(x)\rbrack \nabla\wedge(\eta H)||_{\mathcal{L}^q(B_R)}+C||\tilde{\mathbb{F}}||_{\mathcal{L}^q(B_R)},
\end{equation}
where $q=p^*$ and we take $p=2$ here, hence $q=2^*=6>2$ since $2^*=\frac{2n}{n-2}$ and $n=3$ in this paper.

Let $R$ be small enough so that $C||\varepsilon^{-1}(x_0)-\varepsilon^{-1}(x)||_{\mathcal{L}^{\infty}(B_R)}<<1$, then
\begin{equation}
\begin{aligned}\label{equ:2.7:curl}
&||\nabla\wedge(\eta H)||_{\mathcal{L}^q(B_R)}\leq  C||\tilde{\mathbb{F}}||_{\mathcal{L}^q(B_R)}\leq C(||\omega||_{\mathcal{L}^q(B_R)}+||\mathbb{F}||_{\mathcal{L}^q(B_R)})\\
\leq & C(||\mathbb{G}||_{\mathcal{L}^2(B_R)}+||\mathbb{F}||_{\mathcal{L}^q(B_R)})\leq \frac{C}{R-r}(||\nabla H||_{\mathcal{L}^2(B_R)}+||H||_{\mathcal{L}^q(B_R)})\leq \frac{C}{R-r}||H||_{\mathcal{H}^1(B_R)}.
\end{aligned}
\end{equation}

Let $H=\nabla\wedge A+\nabla\varphi$, and then we have that 
$$
||\nabla\wedge  A||_{\mathcal{W}^{1,q}(B_r)}\leq C(||\nabla\wedge (\eta H)||_{\mathcal{L}^q(B_R)}+||H||_{\mathcal{L}^q(B_R)})\leq \frac{C}{R-r}||H||_{\mathcal{H}^1(B_R)}.
$$
Let $\tilde{\psi}$ be the test function in $B_r$, such that $\tilde{\psi}\equiv1$ in $B_{r_1}$, where $0<r_1<r$ and $\tilde{\psi}=0$ on $\partial B_r$, then for any $k=1,2,3$, the equation (\ref{equ:varphi3.1}) leads to
\begin{align*}
 &\int_{B_r(x_0)}\mu(x_0)\nabla(\eta\partial_k\varphi)\cdot\nabla\tilde{\psi} dx\\
=&\int_{B_r(x_0)}[\mu(x_0)-\mu(x)]\nabla(\eta\partial_k\varphi)\cdot\nabla\tilde{\psi}{\rm d}x+\int_{B_r(x_0)}\mathbb{G}\tilde{\psi}{\rm d}x+\sum_{i=1}^3\int_{B_r(x_0)}\mathbb{F}_i\partial_i\tilde{\psi} {\rm d}x,
\end{align*}
where
$$
\mathbb{G}=-\mu\partial_k(\nabla\wedge A)\cdot\nabla\eta-\partial_k(\mu)\nabla\wedge A\cdot\nabla\eta-\partial_k\mu\nabla\varphi\cdot\nabla\eta-\mu\nabla(\partial_k\varphi)\cdot\nabla\eta
$$
and
$$
\sum_{i=1}^3\mathbb{F}_i\partial_i\tilde{\psi}=-\mu\eta\partial_k(\nabla\wedge A)\cdot\nabla\tilde{\psi}-\eta\partial_k\mu\nabla\wedge A\cdot\nabla\tilde{\psi}-\eta\partial_k\mu\nabla\varphi\cdot\nabla\tilde{\psi}+\mu\partial_k\varphi\nabla\eta\cdot\nabla\tilde{\psi}.
$$
Similarly, for any $k=1$, $2$, or $3$, when $r$ is sufficiently small, we have
\begin{align*}
||\nabla(\eta\partial_k\varphi)||_{\mathcal{L}^q(B_r)}\leq& C(||\mathbb{G}||_{\mathcal{L}^2(B_r)}+||\mathbb{F}||_{\mathcal{L}^q(B_r)})\\
 \leq& C(||\partial_k(\nabla\wedge A)||_{\mathcal{L}^q(B_r)}+||\partial_k(\nabla\varphi)||_{\mathcal{L}^2(B_r)}+||\nabla\varphi||_{\mathcal{L}^q(B_r)})\\
 \leq&\frac{C}{(R-r)(r-r_1)}||H||_{\mathcal{H}^1(B_R)}.
\end{align*}
%Hence it proves the inequality \eqref{2.9}. Now we can follow the exact three steps shown in the case of $\mu=Id^{3\times3}$ to show that $H$ satisfies the %desired $\mathcal{W}^{1,p}$ estimate.
% %%%%%%%%%%%%%%
%\begin{equation}\label{equation2}
%\begin{aligned}
%||\nabla H||_{\mathcal{L}^q(B_r)}\leq & ||\nabla(\eta H)||_{\mathcal{L}^q(B_R)}\leq ||\nabla\wedge(\eta H)||_{\mathcal{L}^q(B_R)}+||{\rm div}(\eta %H)||_{\mathcal{L}^q(B_R)}\\
%\leq & \frac{C}{R-r}||H||_{\mathcal{H}^1(B_R)}+\frac{C}{R-r}||H||_{\mathcal{L}^q(B_R)}\leq \frac{C}{R-r}||H||_{\mathcal{H}^1(B_R)}.
%\end{aligned}
%\end{equation}
and so we obtain that
\begin{equation}\label{lqbdd}
||\nabla H||_{\mathcal{L}^q(B_r)}\leq \frac{C}{R-r}||H||_{\mathcal{H}^1(B_R)}.
\end{equation}

  The following argument improves the inequality \eqref{lqbdd} and gives that $\nabla H\in \mathcal{L}^p(B_{\frac{R}{2}})$, for every $p \geq 2$.  Suppose we replace $2$ and $2^*$ by $l$ and $l^*$ subsequently in the above argument, and redefine $l^*:=\infty$ if $l^*<0$, then we can get that
 \begin{equation}\label{2.9}
 ||\nabla H||_{\mathcal{L}^{l^*}(B_r)}\leq \frac{C}{R-r}||H||_{\mathcal{W}^{1,l}(B_R)}.
 \end{equation}
 Now let us introduce a new Ball with the radius between $\frac{R}{2}$ and $R$ as follows 
 \[\frac{R}{2}<R_1<R. \]
 For any $2\leq q\leq 6=2^*$,
 $$
 ||\nabla H||_{\mathcal{L}^q(B_{\frac{R}{2}})}\leq  C||\nabla H||_{\mathcal{L}^6(B_{\frac{R}{2}})} \leq \frac{C}{R}||H||_{\mathcal{H}^1(B_R)}.
 $$
% then we can always find an integer $k$ such that, any $p\geq 2$ satisfies the following inequality
 For any $p\geq 6$, by the estimate \eqref{2.9}, we get that
% \begin{equation}\label{sequnceBall}
% 2^{\overbrace {\left( {*,*, \cdots ,*} \right)}^{k-1}}<p\leq 2^{\overbrace {\left( {*,*, \cdots ,*} \right)}^{k}}.
% \end{equation}
% Actually, for $n=3$ which is considered in this paper, $k\leq 2$. \bs{It does not work for $p$ huge in \ref{sequnceBall}, and further we did not use the left hand side of the above inequality anywhere and it does not make sense. Is that $ \leq $ but not $<$? }
%
 %After using the equation \eqref{2.9} $k$ times and it gives that
%\begin{equation}
%||\nabla H||_{\mathcal{L}^p(B_{\frac{R}{2}})}\leq C||\nabla H||_{\mathcal{L}^{\overbrace { 2^{\left({*,*, \cdots ,*} \right)}}^{k}}(B_{R_k})}\leq\cdots \leq \frac{C||H||_{\mathcal{H}^1}}{(R-R_1)(R_1-R_2)\cdots(R_{k-1}-R_k)},
%\end{equation} 
\begin{equation}
||\nabla H||_{\mathcal{L}^p(B_{\frac{R}{2}})}\leq C||\nabla H||_{\mathcal{L}^{\infty}(B_{\frac{R}{2}})}\leq  \frac{C||\nabla H||_{\mathcal{L}^6(B_{R_1})}}{R_1-\frac{R}{2}} \leq \frac{C||H||_{\mathcal{H}^1}}{(R-R_1)(R_{1}-\frac{R}{2})}.
\end{equation} 
Hence $$ ||\nabla H||_{\mathcal{L}^p(B_{\frac{R}{2}})}\leq C||H||_{\mathcal{H}^1(B_R)}.$$

\medskip
{\bf Step 2.} \emph{Boundary regularity.} Let $x_0\in\partial\Omega$, and by taking the assumption $\partial\Omega$ is $\mathcal{C}^{1,1}$, we can always find a ball $B$ centred at $x_0$ and an orthogonal coordinates transformation $\Phi\in \mathcal{C}^{1,1}(B;\mathrm{R}^3)$ such that  $(y_1,y_2,y_3)=(\Phi_1(x),\Phi_2(x),\Phi_3(x))$, and that $\Phi(B\cap\Omega)=\{y_3>0\}\cap B_R(0)$ in the new coordinates. Let upper and lower indices denote contravariant and covariant components respectively, and
let the components of tensors be marked by "$ \tilde \quad$"  if they are expressed
in the $y_i$-coordinate system, whereas unmarked components refer to the cartesian
coordinates. Then the following transformation rules are valid, see \cite{gw} for more detail.\\
Let
$$\mathscr{E}_i=\sum_kE_k\frac{\partial x_k}{\partial \Phi_i},\quad \mathscr{H}_i=\sum_k H_k\frac{\partial x_k}{\partial \Phi_i}, \quad \tilde{\varepsilon}^{ij}=\sum_{k,l}\varepsilon^{kl}\frac{\partial \Phi_i}{\partial x_k}\frac{\partial \Phi_j}{\partial x_l},\quad \tilde{\mu}^{ij}=\sum_{k,l}\mu^{kl}\frac{\partial \Phi_i}{\partial x_k}\frac{\partial \Phi_j}{\partial x_l} $$
with $i,j,k,l \in \{1,2,3\}$.\\
Let us set that $\mathscr{F}:=\nabla \wedge E$, and $\mathscr{G}:=\nabla \wedge H$ then we have that

\begin{align}
\sum_k \mathscr{F}_k\frac{\partial x_k}{\partial \Phi_i}&=D \cdot 
(\tilde{\nabla}\wedge \mathscr{E})_i,\\
\sum_k \mathscr{G}_k\frac{\partial x_k}{\partial \Phi_i}&=D \cdot (\tilde{\nabla}\wedge \mathscr{H})_i,
\end{align}
where $D={\rm det} (\frac{\partial \Phi_i(x)}{\partial x_j})$ for $i,j=1,2,3$ and $\tilde{\nabla}=(\partial_{y_1},\partial_{y_2},\partial_{y_3})$.

%\bs{The following equation is correct ? $$\tilde{\nabla}\wedge(\nabla \Phi E)=iw \mu_0 (\nabla \Phi) \mu(\nabla \Phi)^T (\nabla \Phi) H$$, which gives that %$$(\nabla \Phi)^{-1}\tilde{\nabla}\wedge(\nabla \Phi E)=iw\mu_0\mu {H}$$, but this is correct? $$(\nabla \Phi)^{-1}\tilde{\nabla}\wedge(\nabla \Phi E) =? \nabla %\wedge E. $$}

Then we can see that $\mathscr{E}$ and $\mathscr{H}$ satisfy the following equations
\begin{equation}\label{equ:tilde E and tilde H}
\begin{cases}
\tilde{\nabla}\wedge\mathscr{E}=iw\mu_0\tilde{\mu}\mathscr{H}\\
\tilde{\nabla}\wedge\mathscr{H}=-iw\varepsilon_0\tilde{\varepsilon}\mathscr{E}
\end{cases}
\end{equation}
with the boundary conditions on $\{y_3=0\}$ that $\mathscr{E}_1=\mathscr{E}_2=\mathscr{H}_3=0$. Moreover,  we notice that $\nabla\Phi\in \mathcal{W}^{1,\infty}$ which assures that $\tilde{\varepsilon}$ and $\tilde{\mu}$ satisfy all the assumptions which $\varepsilon$ and $\mu$ satisfy.

Then using the half balls in place of the balls, we can follow the exact procedure of {\bf Step 1} to derive the estimates near the boundary. One can treat the boundary integration terms exactly same as the terms we have analyzed in {\bf Step 1} (for $\tilde{\nabla}\wedge\mathscr{H}$) and {\bf Step 4} (for $\tilde{\nabla}\cdot H$) in the proof of Proposition \ref{thm:1}. Hence we omit the details here for the shortness, and we also refer\ucite{ChenWu} and\ucite{Giaquinta} for more detail.
\medskip

{\bf Step 3.} \emph{Global regularity}.
After applying the interior estimates ({\bf Step 1} of this proof), boundary estimates ({\bf Step 2} of this proof) and the standard finite covering technique ({\bf Step 5} of this proof of Proposition \ref{thm:1}) as well as the result of Proposition \ref{thm:1}, one can easily complete the proof of this theorem.

\end{proof}

Next we give the proof of Theorem \ref{thm:inw1p}, in which we obtain the $\mathcal{W}^{1,p}$ estimate for every $p \geq 2$.
\begin{proof}[Proof of the theorem \ref{thm:inw1p}]
Using the interpolation theorem, then one only needs to prove that
$$
||\nabla H||_{\mathcal{L}^{p}(\Omega)}\le C(||H||_{\mathcal{L}^2(\Omega)}+||J_e||_{\mathcal{L}^p(\Omega)}+||\nabla\wedge G||_{\mathcal{L}^p(\Omega)}+||J_m||_{\mathcal{L}^{p}(\Omega)}+||{\rm div}J_m||_{\mathcal{L}^{p}(\Omega)}).
$$
Let us define the terms $\tilde{H}$, $\tilde{E}$, $\tilde{J}_e$ and $\tilde{J}_m$ as in the proof of Theorem \ref{thm:jejm}. Then we only need to prove that
$$
||\nabla \tilde{H}||_{\mathcal{L}^{p}(\Omega)}\le C(||\tilde{H}||_{\mathcal{L}^2(\Omega)}+||\tilde{J}_e||_{\mathcal{L}^p(\Omega)}+||\tilde{J}_m||_{\mathcal{L}^{p}(\Omega)}+||{\rm div}\tilde{J}_m||_{\mathcal{L}^{p}(\Omega)}).
$$
We follow the similar approach as given in the proof of Proposition \ref{thm:2}. Because of the source terms,  $G_i$ and $F_i^{\alpha}$ are now replaced by
\begin{align*}
&\int_{B_R}\tilde{\mathbb{G}}_i\varphi^idx=\int_{B_R}k_1^2\mu \tilde{H}\cdot\eta\varphi dx-\int_{B_R} \varepsilon^{-1}(x)(\nabla\wedge \tilde{H}-\tilde{J}_e)\cdot D\eta\wedge\varphi(x)dx-\int_{B_R}k_2\tilde{J}_m\cdot\eta\varphi dx\\
&\int_{B_R}\tilde{\mathbb{F}}^{\alpha}_iD_{\alpha}\varphi^i dx=\int_{B_R}\varepsilon^{-1}(x)(D\eta\wedge \tilde{H}-\tilde{J}_e)\cdot\nabla\wedge\varphi(x) dx.
\end{align*}

Following the same argument as the proof to the estimate \eqref{equ:2.7:curl}, we can have the estimate of $\nabla\wedge (\eta H)$. At the same time, ${\rm div}(\eta \tilde{H})$ is estimated through equation \eqref{equ:4.4a} by applying the technics developed in {\bf Step 4} of the proof of Proposition \ref{thm:2}, and the inhomogeneous terms are treated using the exact same method used in the proof of Theorem \ref{thm:jejm}. % and the boundary estimate is done as in the proof of theorem \ref{thm:2}.
Then we can easily obtain the $\mathcal{W}^{1,p}$-estimates. Here we need to use the following different estimate to bound the integration on the boundary
\begin{align*}
&\int_{\partial\Omega\cap B_R(y_0)}\eta^2\overline{\partial_k\varphi}\partial_k\tilde{J}_m\cdot e_3{\rm d}\sigma\\ \leq&||\eta^2\partial_k\varphi||_{\mathcal{W}^{1-\frac{1}{p},p}(\partial\Omega)} ||\partial_k(\tilde{J}_m\cdot\boldsymbol{\nu})||_{\mathcal{W}^{-(1-\frac{1}{p}),q}(\partial\Omega)}\\
\leq&C(||\nabla\varphi||_{\mathcal{W}^{1,p}(\Omega)\cap B_r(y_0)} +||\nabla\varphi||_{\mathcal{L}^{p}(\Omega\cap B_R(y_0))}) ||\tilde{J}_m\cdot\boldsymbol{\nu}||_{\mathcal{W}^{\frac{1}{p},q}(\partial\Omega)},
\end{align*}
where $\frac{1}{p}+\frac{1}{q}=1$.

Since the rest of the proof is similar, we omit them for brevity.
\end{proof}

                                  %%%%%%%%%%%%%%%%%%%%%%%%%%%%%%%%%%%%%%%%%%%%%%%%%%%%%%%%%%%%%%%%%%%%%%%%%%%%%%%%%%%%%%%%%%%%%%%%%%%%%%%%%%%

                                  %%%%%%%%%%%%%%%%%%%%%%%%%%%%%%%%%% Proof of Theorem 3 %%%%%%%%%%%%%%%%%%%%%%%%%%%%%%%%%%%%%%%%%%%%%%%%%%%%%

                                  %%%%%%%%%%%%%%%%%%%%%%%%%%%%%%%%%%%%%%%%%%%%%%%%%%%%%%%%%%%%%%%%%%%%%%%%%%%%%%%%%%%%%%%%%%%%%%%%%%%%%%%%%%%%

\subsection{Proof of Theorem \ref{thm:15}}
%We first develop the following proposition which gives the almost everywhere $\mathcal{C}^{1,\alpha}$ regularity for the homogeneous Maxwell's equations. We then generalize the result from homogeneous to nonhomogeneous and give the proof of the theorem \ref{thm:15}. We use, in this part, the Campanato space which is introduced in definition \ref{def:14} in Appendix.
%\begin{prop}\label{thm:6}Let $\mu(x)\in \mathcal{W}^{1,3+\delta}(\Omega)^{3 \times 3}$, where $\delta>0$. Assume $k_1$ is not a Maxwell eigenvalue.
%Suppose $H \in \mathcal{H}^1(\Omega)$ satisfies the equations
%\eqref{eq1},
%%with the boundary condition $(\varepsilon^{-1}\nabla\wedge H)\wedge \boldsymbol{\nu}=E\wedge\boldsymbol{\nu}=0$ on $\partial\Omega$,
%where $\varepsilon^{-1}(x)\in \mathcal{C}^{\beta}(\Omega)^{3 \times 3}$, then there exists an open set $\Omega_h\subset\Omega$ and $\alpha \in (0,1)$, such that $meas(\Omega\backslash\Omega_h)=0$, and for any $x_0\in\Omega_h$, there exists $r>0$ and $C>0$ such that
%\begin{equation}
%||H||_{\mathcal{C}^{1,\alpha}(B_r(x_0))}\le C(1+\|H\|_{\mathcal{L}^2(\Omega)}).
%\end{equation}
%\end{prop}
%
% %As before let us first consider the case $\mu=Id^{3\times3}$, and the opposite case that $\mu\neq Id^{3\times3}$ will be studied later.

 Before giving the proof, we first address the following three lemmas. The first lemma develops the local energy estimate of the solutions with constant coefficients.
\begin{lem}\label{lem7}
If $V$ satisfies
\begin{equation}\label{eqv}
\nabla\wedge(\varepsilon^{-1}(x_0)\nabla\wedge V)=k^2_1\mu(x_0)V \quad \text{and} \quad {\rm div}(\mu(x_0)V)=0\quad \mbox{in} \, \, B_R(x_0),\end{equation}
%with boundary condition  $(\varepsilon^{-1}(x_0)\nabla\wedge V)\wedge\boldsymbol{\nu}=0$, or $V\wedge\boldsymbol{\nu}
%=0$ on $\partial\Omega$,
then
\begin{equation}\label{equ:2.13v}
\int_{B_r(x_0)}|\nabla V|^2dx\leq C(\frac{r}{R})^n\int_{B_R(x_0)}|\nabla V|^2dx, \quad \text{for all} \quad 0\leq r \leq R
\end{equation}
where $n=3$ is the dimension.
\end{lem}

\begin{proof}
Since $\partial_kV$ satisfies the same equation \eqref{eqv}, we only need to prove
\begin{equation*}\int_{B_r(x_0)}|V|^2dx\leq C(\frac{r}{R})^3\int_{B_R(x_0)}|V|^2dx, \qquad \text{for all} \quad 0\leq r \leq R  \end{equation*}
For the simplicity, we assume that $\mu(x_0)$ is an identity matrix, for the non-identity matrix case we apply the argument in {\bf Step 3} of the proof of Proposition \ref{thm:1} to estimate ${\rm div} V$. 

Let us use the substitution $y=\frac{x-x_0}{R}$ in the previous equation and write 
$$
\tilde{V}=\tilde{V}(y)=V(x),
$$ 
then $B_R(x_0)$ becomes $B_1({\rm O})$. Let $\eta \in \mathcal{C}^{\infty}_0(B_1({\rm O}))$ and $\eta\equiv 1$ in $B_r({\rm O})$, $|\nabla\eta|\leq \frac{C}{1-r}$ and $0 \leq \eta \leq 1$. Let us multiply the equation
\eqref{eqv} by $\eta^2 \overline{\tilde{V}}$ and integrate it by part,
\begin{equation}
\begin{aligned}
&\int_{B_1({\rm O})}\varepsilon^{-1}(x_0)\nabla\wedge\tilde V\cdot \overline{\nabla\wedge(\eta^2\tilde V)}dy=\int_{B_1({\rm O})}k^2_1\eta^2\tilde V\cdot\overline{\tilde V}dy, \\
&\int_{B_1({\rm O})}\eta^2|\nabla\wedge\tilde V|^2dy\leq C\int_{B_1({\rm O})} \eta^2 |\tilde V|^2 dy+C\int_{B_1({\rm O})}|\nabla\eta|^2|\tilde V|^2dy.
\end{aligned}
\end{equation}
It gives that
$$ 
\int_{B_r({\rm O})}|\nabla\wedge\tilde V|^2dy\leq \int_{B_1({\rm O})}\eta^2|\nabla\wedge\tilde V|^2dy \leq C(r)\int_{B_1({\rm O})} |\tilde V|^2 dy.
$$

Furthermore, because ${\rm div}(\tilde V)=0$ in $B_1({\rm O})$, we have that 
$$
\int_{B_r({\rm O})} |\nabla\tilde V|^2 dy\leq C(r)\int_{B_1({\rm O})} |\tilde V|^2 dy.
$$
%Let $2k>n$ and $\frac{R}{2}<R_k<R_{k-1}<\cdots<R_1<R$, by induction, we get that
Therefore, by the induction, we get that
\begin{equation}
\int_{B_{\frac{1}{2}({\rm O})}} |\nabla^{(2)}\tilde V|^2 dy\leq C\int_{B_{\frac{3}{4}({\rm O})}} |\nabla\tilde V|^2 dy\leq C\int_{B_1({\rm O})} |\tilde V|^2 dy.
\end{equation}
Hence 
\begin{equation}\label{3.35x}
||\tilde V||^2_{\mathcal{L}^{\infty}(B_{\frac{1}{2}}({\rm O}))}\leq C\sum_{k=0}^2 \int_{B_{\frac{1}{2}}({\rm O})}|\nabla^{(k)}\tilde V|^2dy\leq C\int_{B_1({\rm O})} |\tilde V|^2 dy.
\end{equation}
If $0<\frac{r}{R}<\frac{1}{2}$, then \eqref{3.35x} leads to
\begin{equation}
\int_{B_\frac{r}{R}({\rm O})} |\tilde V|^2 dy \leq C\int_{B_1({\rm O})} |\tilde V|^2 dy\cdot |B_{\frac{r}{R}}({\rm O})|=C(\frac{r}{R})^3\int_{B_1({\rm O})} |\tilde V|^2 dy.
\end{equation}
It follows that if $0<r<\frac{R}{2}$, then
\begin{equation}
\int_{B_r(x_0)} |V|^2 dx=R^3\int_{B_{\frac{r}{R}({\rm O})}}|\tilde{V}|^2dy \leq C(\frac{r}{R})^3R^3\int_{B_1({\rm O})}|\tilde{V}|^2dy=C(\frac{r}{R})^3\int_{B_R(x_0)} |V|^2 dx,
\end{equation}
where the constant $C$ does not depend on $r$ and $R$.

On the other hand, if $\frac{R}{2}\leq r <R$, then we easily obtain that.

\begin{equation}
\int_{B_r(x_0)} |V|^2 dx \leq 2^3(\frac{r}{R})^3\int_{B_R(x_0)} |V|^2 dx.
\end{equation}
\end{proof}

\begin{rem}We actually proved that if $V$ satisfies
\eqref{eqv},
then
\begin{equation}\label{3.21a}
\int_{B_r(x_0)}|V|^2dx\leq C(\frac{r}{R})^3\int_{B_R(x_0)}|V|^2dx, \quad \text{for all} \quad 0\leq r \leq R.
\end{equation}
%where $n=3$ is the dimension.
\end{rem}

Based on Lemma \ref{lem7}, we further have the following lemma for the solutions of the equation with constant coefficients.
\begin{lem}\label{lem9}
If $V$ satisfies \eqref{eqv} in $B_R(x_0)$, then
\begin{equation}\label{2.19}
\int_{B_r(x_0)}|V-\bar{V}_{x_0,r}|^2dx\le C(\frac{r}{R})^{n+2}\int_{B_R(x_0)}|V-\bar{V}_{x_0,R}|^2 dx,
\end{equation}
where $\bar{V}_{x_0,r}:=\frac{1}{|B_r(x_0)|}\int_{B_r(x_0)}V(x)dx$ and $|B_r(x_0)|$ is the volume of the ball $B_r(x_0)$, and $n=3$ is the dimension.
\end{lem}

\begin{proof} Notice that $\partial_kV$ satisfies the same equation as $V$. If $r<\frac{R}{2}$,  then by the Poincare inequality and Lemma \ref{lem7}, we get that
\begin{equation}\label{2.20}
\int_{B_r(x_0)}|V-\bar{V}_{x_0,r}|^2dx\le Cr^2\int_{B_r(x_0)}|\nabla V|^2dx \leq C\frac{r^{n+2}}{R^n}\int_{B_R(x_0)}|\nabla V|^2 dx.
\end{equation}
Now let us introduce a cut-off function $\eta$ and multiply on both sides of the given equation by $\eta^2(V-\bar{V}_{x,R})$. After an integration by part, we get that
\begin{equation}
Cr^2(\frac{r}{R})^n\int_{B_R(x_0)}|\nabla V|^2dx\le Cr^2(\frac{r}{R})^n\frac{1}{R^2}\int_{B_R(x_0)}|V-\bar{V}_{x,R}|^2dx.
\end{equation}
and it is obvious in the case of $r>\frac{R}{2}$.
\end{proof}

The next lemma is classical and the previous methods used could be referred to
as Caccioppoli estimates. This lemma is used to absorb the error terms described by the small parameter $\delta$.
\begin{lem}\label{lem6}
If $\Phi(\rho)$ is monotone increasing and satisfies that
$$
\Phi(\rho)\leq A\lbrack(\frac{\rho}{R})^{\tau_1}+\delta\rbrack \Phi(R)+BR^{\tau_2},
$$
for every $\rho$, such that $ 0<\rho< R\leq R_0$, and $\tau_2<\tau_1$, then let $\delta$ be small enough such that there exists $\tau_2<\nu<\tau_1$, and gives that
$$
\Phi(\rho)\leq C\lbrack(\frac{\rho}{R})^{\nu}\Phi(R)+B\rho^{\tau_2}\rbrack.
$$
\end{lem}
\begin{proof}
For all $\theta\in(0,1)$, $\Phi(\theta R)\leq A\theta^{\tau_1}\lbrack 1+ \theta^{-\tau_1}\delta\rbrack\Phi(R)+BR^{\tau_2}$. Let $\tau_2<\nu<\tau_1$. Suppose that $2A\theta^{\tau_1}=\theta^{\nu}$, and let $\delta$ be small enough such that $\delta \theta^{-\tau_1}\leq1$, then it gives that
\[
\Phi(\theta R)\le 2A\theta^{\tau_1}\Phi(R)+BR^{\tau_2}\le \theta^{\nu} \Phi(R)+BR^{\tau_2}.
\]
Let us take $k$ for any $0<r<R$, such that $\theta^{1+k}R<r<\theta^kR$, then it follows that
\[
\Phi(r)\le\Phi(\theta^kR)\le\theta^{k\nu}\Phi(R)+B\theta^{(k-1)\tau_2}R^{\tau_2}\frac{1}{\theta^{\tau_2}-\theta^{\nu}}\le C(\frac{r}{R})^{\tau_2}((\frac{r}{R})^{\nu-\tau_2}\Phi(R)+BR^{\tau_2}),
\]
and so it leads to the conclusion.
\end{proof}

Now we are ready to prove Theorem \ref{thm:15}.

\begin{proof}[Proof of theorem \ref{thm:15}]

{\bf Step 1} \, Consider $(V, \tilde{E})$ be the solution of the system:

\begin{equation}\left\{
\begin{array}{rl}
\nabla \wedge \tilde{E}-i \omega \mu_0 \mu(x_0)V =0 \quad &\text{in}\, B_R(x_0),\\
\nabla \wedge V + i \omega \varepsilon_0 \varepsilon(x_0)\tilde{E} =0 \quad &\text{in}\, B_R(x_0),\\
\tilde{E}\wedge \boldsymbol{\nu} = E \wedge \boldsymbol{\nu} \quad &\text{on} \, \partial B_R(x_0).
\end{array} \right.
\end{equation}
Note that $J_e=J_m=0$, then $V$ satisfies the equation (\ref{eqv}) in $B_R(x_0)$, and the following boundary condition
\begin{align*}
(\varepsilon^{-1}(x_0)\nabla \wedge V)\wedge  \boldsymbol{\nu} = -i \omega \varepsilon_0\tilde{E}\wedge  \boldsymbol{\nu} = -i\omega\varepsilon_0 E\wedge  \boldsymbol{\nu}
=(\varepsilon^{-1}(x)\nabla \wedge H) \wedge  \boldsymbol{\nu}   \qquad \text{on} \, \partial B_R(x_0)
\end{align*}

The existence of $V$ is guaranteed by Corollary 4.19 of the book \cite{pm}, since $k_1$ is not a Maxwell eigenvalue.

Let $W=H-V$. Let $B_R:=B_R(x_0)$. Let $B_r:=B_r(x_0)$. Then
\begin{equation}\label{3.36x}
\int_{B_r}|\nabla H|^2dx \le 2\int_{B_r}|\nabla W|^2dx+2\int_{B_r}|\nabla V|^2dx.
\end{equation}
Because the coefficients of the equations  \eqref{eqv} are constants, by Lemma \ref{lem7}, we can derive the inequality as follows
\begin{equation}\label{E3.54}
\int_{B_r}|\nabla H|^2dx\le C(\frac{r}{R})^3\int_{B_R}|\nabla V|^2dx+C\int_{B_R}|\textrm{div} W|^2dx+C\int_{B_R}|\nabla\wedge W|^2dx.
\end{equation}

{\bf Step 2}
Note that by the boundary condition of $V$ and Proposition \ref{thm:2}, $H$ and $V$ are in $\mathcal{W}^{1,p}(B_R)$. It provides that $W\in \mathcal{W}^{1,p}(B_R)$, with the boundary condition $W\wedge\boldsymbol{\nu}=0$ on $\partial B_{R}$. Further, $W$ satisfies the following equations in $B_R$,
\begin{equation}\label{thm5p1}
\nabla \wedge [\varepsilon^{-1}(x_0)\nabla \wedge W]=k_1^2\mu(x_0)W+k_1^2(\mu(x)-\mu(x_0))H+\nabla\wedge[(\varepsilon^{-1}(x_0)-\varepsilon^{-1}(x))\nabla \wedge H],
\end{equation}
and
\begin{equation}\label{thm5p2}
{\rm div}(\mu(x_0)W)={\rm div}((\mu(x_0)-\mu(x))H),
\end{equation}
with the boundary condition
\begin{equation}\label{thm5pb1}
(\varepsilon^{-1}(x_0)\nabla\wedge W)\wedge  \boldsymbol{\nu}=((\varepsilon^{-1}(x_0)-\varepsilon^{-1}(x))\nabla \wedge H)\wedge  \boldsymbol{\nu} \qquad \text{on} \, \partial B_R.
\end{equation}
Multiply $\overline{W}$ on the both sides of the equation (\ref{thm5p1}) and integrate it by parts, then
\begin{align*}
&  \lambda \int_{B_R}|\nabla\wedge W|^2dx \\
& \leq \int_{B_R}(\varepsilon^{-1}(x_0)\nabla \wedge W ) \nabla \wedge\overline{W} dx\\
&  = \int_{B_R}\overline{W}(\nabla \wedge (\varepsilon^{-1}(x_0)\nabla \wedge W ))dx+\int_{\partial B_R}\overline{W}(\varepsilon^{-1}(x_0)\nabla \wedge W )\wedge  \boldsymbol{\nu} ds \\
&  = \int_{B_R}k_1^2 \mu(x_0)W\overline{W}dx+\int_{B_R}k_1^2(\mu(x)-\mu(x_0))H\overline{W}dx\\
& +\int_{B_R}((\varepsilon^{-1}(x_0)-\varepsilon^{-1}(x))\nabla \wedge H) \nabla \wedge \overline{W}dx \\
&-\int_{\partial B_R}((\varepsilon^{-1}(x_0)-\varepsilon^{-1}(x))\nabla \wedge H ) \wedge  \boldsymbol{\nu} \overline{W}ds \\
& +  \int_{\partial B_R}\overline{W}(\varepsilon^{-1}(x_0)\nabla \wedge W)\wedge  \boldsymbol{\nu} ds.
\end{align*}
By the condition (\ref{thm5pb1}), we know that if $\mu \in C^{\alpha}(\Omega)$ and $\varepsilon \in C^{\alpha}$, then

\begin{align}
& \int_{B_R}|\nabla \wedge W|^2dx\label{3.41x2} \\
& \leq C \int_{B_R}|W|^2 dx+CR^{2\alpha}\int_{B_R}|H|^2dx+CR^{2\alpha}\int_{B_R}|\nabla \wedge H|^2dx\nonumber\\
& \leq CR^{2\alpha}\int_{B_R}(|H|^2+|\nabla \wedge H|^2)dx+C R^2||W||^2_{\mathcal{H}^1}\nonumber.
\end{align}

{\bf Step3 } Note that by equation (\ref{thm5p2}),we can obtain that
\begin{equation}
{\rm div}W={\rm div}((I-\mu(x_0))W)+{\rm div}((\mu(x_0)-\mu(x))H).
\end{equation}
If $|I-\mu(x_0)|< \delta_0$, $\mu \in C^{\alpha}$, and $\mu \in W^{1,3+\delta}$, then
\begin{align*}
 \int_{B_R}|{\rm div} W|^2dx  & \leq C\delta_0\int_{B_R}|\nabla W|^2dx+CR^{2\alpha}\int_{B_R}|\nabla H|^2dx +C\int_{B_R}|\nabla \mu|^2|H|^2dx,
\end{align*}
%If $\beta$ small enough,\bs{we do not use $\beta$ these part ?} then
hence it can be written that 
\begin{equation*}
\int_{B_R}|{\rm div} W|^2dx\leq C\delta_0\int_{B_R}|\nabla \wedge W|^2dx+CR^{2\alpha}\int_{B_R}|\nabla H|^2dx+CR^{\frac{2\delta}{9+3\delta}}\int_{B_R}|\nabla H|^2dx. 
\end{equation*}
Let $\alpha^*={\rm min}\{\alpha,\frac{\delta}{9+3\delta}\}$, then by \eqref{3.41x2} (we write $\alpha$ instead of $\alpha^*$ in the following equations for the convenience )
\begin{equation}\label{E3.60}
\int_{B_R}|\nabla W|^2dx \leq CR^{2\alpha}||H||^2_{\mathcal{H}^1}+CR^2||W||^2_{\mathcal{H}^1}.
\end{equation}

{\bf Step 4} By \eqref{E3.54}, we now know that
\begin{align*}
\int_{B_r}|\nabla H|^2dx  \leq C [(\frac{r}{R})^3+R^{2\alpha}]\int_{B_R}(|\nabla H|^2+|H|^2)dx+CR^2||W||^2_{\mathcal{H}^1},
\end{align*}
and similarly,
\begin{align*}
\int_{B_r}|H|^2dx \leq&C(\frac{r}{R})^3\int_{B_R}| V|^2dx+C\int_{B_R}|W|^2dx \\
\leq & C [(\frac{r}{R})^3+R^{2\alpha}]\int_{B_R}(|\nabla H|^2+|H|^2)dx+CR^2||W||^2_{\mathcal{H}^1}.
\end{align*}

Let's set that
$$
\Phi(r)=\int_{B_r}(|\nabla H|^2+|H|^2)dx,
$$
and set
$\delta:=R^{2\alpha}$ where $R$ is sufficiently small such that $\delta$ is small. By Lemma \ref{lem6}, if $||\nabla W||^2_{\mathcal{L}^2(B_R)}<CR^{1+\alpha}$ for any $0<R\leq R_0$, then we obtain that
\[
\Phi(r) \leq C\left[(\frac{r}{R})^{ \nu}\Phi(R)+Cr^{3+\alpha}\right],
\]
for any $1+2\alpha <  \nu<3$. It means
$$
\int_{B_r}(|\nabla H|^2+|H|^2)dx\leq C(\frac{r}{R})^{ \nu}\int_{B_R}(|\nabla H|^2+|H|^2)dx+Cr^{3+\alpha}.
$$

Moreover, if $||\nabla H||^2_{\mathcal{L}^2(B_R)}<CR^{\nu}$ for any $0<R\leq R_0$, then by \eqref{E3.60}, we have that
\begin{equation}
\int_{B_R}|\nabla W|^2dx \leq CR^{2\alpha+\nu}||H||^2_{\mathcal{H}^1(B_{R_0})}+CR^{3+2\alpha},\quad \text{for any} \,0<r\leq R\leq R_0
\end{equation}
%hen by \eqref{3.41x2} and \eqref{3.47x2}, for any $0<r\leq R\leq R_0$, we have
%\begin{equation}
%\int_{B_r}|\nabla W|^2dx\leq Cr^{2\alpha}||H||^2_{\mathcal{H}^1(B_{R})}+CR^{1+2\alpha}\|W\|_{\mathcal{H}^1(B_{R})}^{2}\leq CR^{2\alpha+\gamma}||H||^2_{\mathcal{H}^1(B_{R_0})}+CR^{3+2\alpha}.
%\end{equation}

{\bf Step 5}
Note that
\begin{equation}\label{eqforthm5}
\begin{aligned}
\int_{B_r}|\nabla H-(\bar{\nabla H}_{x_0,r})|^2dx \le & \int_{B_r}|\nabla V-(\bar{\nabla H}_{x_0,r})|^2dx+\int_{B_r}|\nabla W|^2dx\\
                                           \le  & \int_{B_r}|\nabla V-(\bar{\nabla V}_{x_0,r})|^2dx+C\int_{B_r}|\nabla W|^2dx,
\end{aligned}
\end{equation}
where $\bar{\nabla H}_{x,r}:=\frac{1}{|B_r(x)|}\int_{B_r(x)}\nabla H(z)dz$.
Moreover,
$$
 \int_{B_R}|(\bar{\nabla H}_{x_0,R})-(\bar{\nabla V}_{x_0,R})|^2dx=|B_R|\left| \frac{1}{|B_R|}\int_{B_R}\nabla W dx\right|^2\leq\int_{B_R}|\nabla W|^2 dx.
$$
After an application of the Lemma \ref{lem9} with $n=3$, we get that for any $0<r\leq R\leq R_0$,
\begin{equation}\label{3.27a}
\begin{aligned}
 & \int_{B_r}|\nabla V-(\bar{\nabla V}_{x_0,r})|^2dx+C\int_{B_r}|\nabla W|^2dx \\
 \le & C(\frac{r}{R})^{5}\int_{B_{R}}|\nabla V-(\bar{\nabla V}_{x_0,R})|^2dx+C\int_{B_{R}}|\nabla W|^2dx \\
 \le & C(\frac{r}{R})^{5}\int_{B_R}|\nabla V-(\bar{\nabla H}_{x_0,R})|^2dx+C\int_{B_R}|\nabla W|^2dx \\
  \le & C(\frac{r}{R})^{5}\int_{B_R}|\nabla H-(\bar{\nabla H}_{x_0,R})|^2dx+C\int_{B_R}|\nabla W|^2dx \\
  \le & C(\frac{r}{R})^{5}\int_{B_R}|\nabla H-(\bar{\nabla H}_{x_0,R})|^2dx +CR^{2\alpha+\nu}||H||^2_{\mathcal{H}^1(B_{R_0})}+CR^{3+2\alpha}, %\\
 % \le  & C(\frac{r}{R})^{n+2}\int_{B_R(x_0)}|\nabla H-(\bar{\nabla H}_{x,R})|^2dx+CR^{\alpha_0+2\alpha}||H||^2_{\mathcal{H}^1(B_{R_0}(x_0))}+CR^{1+\frac{2(p-3)}{p}}\|W\|_{\mathcal{W}^{1,p}(B_R)}^{2},
\end{aligned}
\end{equation}
where the constant $C$ does not depend on $r$ and $R$.
Let
$$
\Phi(r)=\int_{B_r(x_0)}|\nabla H-(\bar{\nabla H}_{x,r})|^2dx,
$$
and choose $\nu$ such that $2\alpha+\nu>3$.  Let
$$
\sigma:=\min(2\alpha+\nu,3+2\alpha,5)\in(3,5),
$$
then by Lemma \ref{lem6},
$$
\int_{B_r}|\nabla H-(\bar{\nabla H}_{x,r})|^2dx \le C(\frac{r}{R})^{\sigma}\int_{B_R}|\nabla H-(\bar{\nabla H}_{x,R})|^2dx+Cr^{\sigma},
$$
and hence
\[
||\nabla H||_{\mathfrak{L}^{2,\sigma}(B_r)} \le C.
\]
By Lemma \ref{lem:16} which is listed in the Appendix, we obtain that
$$
||\nabla H||_{\mathcal{C}^{0,\alpha}(B_r)}\le C,
$$
where $\alpha=\frac{\sigma-n}{2}>0$. Therefore we get that
$$
H \in \mathcal{C}^{1+\alpha}(B_r).
$$

\medskip
{\bf Step 6}. The above steps proved that for any given $x_0\in\Omega$, if there exists $R_0$ such that
\begin{equation}\label{3.52x2}
\|\nabla W\|^2_{\mathcal{L}^2(B_{R}(x_0))}<CR^{1+\alpha}\quad\mbox{and}\quad ||\nabla H||^2_{\mathcal{L}^2(B_R(x_0))}<CR^{\nu},
\end{equation}
for any $0<R\leq R_0$, then we know the solution $H\in\mathcal{C}^{1+\alpha}(B_R(x_0))$. Let $\Omega_h$ be the set of all the points in $\Omega$ such that \eqref{3.52x2} holds. Note that $\|\nabla W\|^2_{\mathcal{L}^2(B_{R}(x_0))}$ and $||\nabla H||^2_{\mathcal{L}^2(B_R(x_0))}$ are continuous with respect to $x_0$, so $\Omega_h$ is open. Following the Lebesgue differentiation theorem, we know that for almost every point in $\Omega$, \eqref{3.52x2} holds, since $|B_R(x_0)|=\frac{4\pi}{3}R^3$ and $(\nabla H, \nabla W)\in (\mathcal{L}^2(\Omega))^2$. Therefore, we have that
$$
meas(\Omega\backslash\Omega_h)=0.
$$

Then we can finish the proof of the theorem 5.
\end{proof}

\subsection{Proof of theorem \ref{thm:E4i} and \ref{them:10a}}

Using the symmetrical structure of the system, we can show a similar interior estimate to the solution $E$. However, one needs to take care of the different boundary condition satisfied by $E$ and $H$, namely the solution $E$ satisfies $E\wedge \boldsymbol{\nu}=G\wedge \boldsymbol{\nu} $ on the boundary and $H$ satisfies 
$(\varepsilon^{-1}\nabla\wedge H)\wedge \boldsymbol{\nu}=(\varepsilon^{-1}J_e)\wedge\boldsymbol{\nu}-k_2 G\wedge\boldsymbol{\nu}$ on the boundary. Hence, in this subsection, we give the detail of the argument to derive the boundary estimates of the solution $E$ of the homogeneous equations in the proof of the propositions (\ref{thm:E4}), and the argument corresponding to the inhomogeneous source terms can be achieved from the proof of theorem \ref{thm:jejm}--\ref{thm:15} by applying the symmetrical structure of the system. 

The regularity results we developed for the solution $E$ of the homogeneous Maxwell's equations are given as below.
\begin{prop}[Regularity for $E$--Homogeneous]\label{thm:E4} Let $\varepsilon(x)\in \mathcal{W}^{1,3+\delta}(\Omega)^{3 \times 3}$, where $\delta>0$.  Suppose $\mu^{-1}(x)\in \mathcal{L}^{\infty}(\Omega)^{3 \times 3}$ and that the condition \eqref{cond1} holds.
If $E$ is a weak solutions of the following equations,
\begin{equation}\label{eqE1}
\nabla\wedge(\mu^{-1}(x)\nabla\wedge E)=k^2\varepsilon(x)E, \qquad {\rm div} (\varepsilon E)=0\qquad\text{in }\Omega,
\end{equation}
with the boundary condition $E\wedge\boldsymbol{\nu}=0$ on $\partial\Omega$, then
$E\in \mathcal{H}^1(\Omega)$ and
\begin{equation}\label{equ:3.56E}
||E||_{\mathcal{H}^1(\Omega)}\le C||E||_{\mathcal{L}^2(\Omega)}.
\end{equation}
Moreover, if $\mu^{-1}(x)\in \mathcal{C}^0(\Omega)^{3\times3}$ and satisfies the condition \eqref{cond1},
then the following inequality holds
\begin{equation}\label{EreguW}
||E||_{\mathcal{W}^{1,p}(\Omega)}\leq C||E||_{\mathcal{L}^2(\Omega)},\qquad p \geq 2,
\end{equation}
where the constant $C$ does not depend on the solutions $E$; and finally assume $k$ is not a Maxwell eigenvalue, and if $\mu^{-1}(x)\in \mathcal{C}^{\beta}(\Omega)^{3 \times 3}$,and $||\varepsilon(x)-I||_{\mathcal{L}^{\infty}(\Omega)^{3 \times 3}} < \delta_0$, then there exists an open set $\Omega_h\subset\Omega$ and $\alpha \in (0,1)$, such that $meas(\Omega\backslash\Omega_h)=0$, and for any $x_0\in\Omega_h$, there exists $r>0$ and $C>0$ such that

\begin{equation} \label{EreguC}
||E||_{\mathcal{C}^{1,\alpha}(B_r(x_0))}\le C(1+||E||_{\mathcal{L}^2(\Omega)}).
\end{equation}
\end{prop}

\begin{rem}
The result (\ref{equ:3.56E}) and (\ref{EreguW}) can be extended to the inhomogeneous case; Similar to the regularity of $H$, the result (\ref{EreguC}) only works with the homogeneous equation ($J_e=J_m=0$) case.
\end{rem}

%\begin{prop}[Regularity for $(E,H)$--Homogeneous]\label{thmEH}
%Let $\mu(x), \varepsilon(x)\in \mathcal{W}^{1,3+\delta}$, where $\delta>0$. Let $\varepsilon(x)$ and $\mu(x)$ satisfy \eqref{cond1}. Then the solution $(E,H)$ %satisfies the estimate that
%$$
%\begin{aligned}
%||H||_{\mathcal{W}^{1,p}(\Omega)}+||E||_{\mathcal{W}^{1,p}(\Omega)}
%\leq& C(||\nabla\wedge G||_{\mathcal{L}^p(\Omega)}+||J_m||_{\mathcal{W}^{1,p}(\Omega)}+||J_e||_{\mathcal{W}^{1,p}(\Omega)}\\
% &+||E||_{\mathcal{L}^2(\Omega)}+||H||_{\mathcal{L}^2(\Omega)}+||(J_m-\nabla\wedge G)\cdot\boldsymbol{\nu}||_{\mathcal{W}^{1-\frac{1}{p},q}(\partial\Omega)}),
%\end{aligned}
%$$
%where $\frac{1}{p}+\frac{1}{q}=1$. Moreover, assume that $k_1$ and $k$ are not Maxwell eigenvalues. Then there exists an open set $\Omega_h\subset\Omega$ and %$\alpha \in (0,1)$, such that $meas(\Omega\backslash\Omega_h)=0$, and for any $x_0\in\Omega_h$, there exists $r>0$ and $C>0$ such that

%\begin{equation}
%||E||_{\mathcal{C}^{1,\alpha}(B_r(x_0))}+||H||_{\mathcal{C}^{1,\alpha}(B_r(x_0))}\le C(1+||E||_{\mathcal{L}^2(\Omega)}+||H||_{\mathcal{L}^2(\Omega)}).
%\end{equation}
%\end{prop}

%The proofs of theorem  \ref{thm:1}--\ref{thmEH} are given in section \ref{sec:3a}.
\begin{proof}[Proof of Proposition \ref{thm:E4}]
The only difference in this case is the property of the terms on the boundary $\partial\Omega$, since $E$ and $H$ satisfy the different boundary conditions. Let us multiply both sides of the equation \eqref{eqE1} by $\overline{E}$, and integrate it by parts, we get that
$$\int_{\Omega}\mu^{-1}(x)(\nabla\wedge E)\cdot\overline{(\nabla\wedge E)}dx+\int_{\partial\Omega}(\mu^{-1}\nabla\wedge E)\cdot \overline{(E\wedge\boldsymbol{\nu})}{\rm d}\sigma=k^2\int_{\Omega}|E|^2dx.$$
Notice that $E\wedge\boldsymbol{\nu}=0$ on $\partial\Omega$, and so
$$\int_{\Omega}|\nabla\wedge E|^2dx \le C\int_{\Omega}|E|^2dx.$$
Based on the estimate of $\nabla\wedge E$, we can show the interior estimate of $|\nabla E|$ exactly as done in the {\bf Step 3} of the proof of Proposition \ref{thm:1}, where we only need to exchange the role of $\varepsilon$ and $\mu$.

Now we are going to derive the boundary estimates. Let $x_0\in\partial\Omega$ and $B_r$ be the ball with the center $x_0$ and radius $r$.
Let $\eta$ be a smooth positive cut-off function such that $\eta\equiv1$ in $B_r$ and $\eta=0$ on $\partial B_R$.Then multiply $\overline{\eta^2E}$ on the both sides of the equation \eqref{eqE1}, after an integration by part,
\begin{equation}\label{3.60x}
\int_{B_R}\mu^{-1}(x)(\nabla\wedge E)\cdot\overline{(\nabla\wedge (\eta^2 E))}dx+\int_{\partial B_R}(\mu^{-1}\nabla\wedge E)\cdot \overline{(\eta^2E\wedge\boldsymbol{\nu})}{\rm d}\sigma=k^2\int_{B_R}|\eta E|^2dx.
\end{equation}
Notice that $\eta^2 E\wedge\boldsymbol{\nu}=0$ on $\partial{B_R}$, so by employing the H\"{o}lder inequality, it is easy to have that
\begin{equation}
\int_{B_r\cap\Omega}|\nabla\wedge E|^2dx \le C(R)\int_{B_R\cap\Omega}|E|^2dx,
\end{equation}
where the constant $C(R)$ depends on $R$ but does not depend on the solutions $E$ and $H$.

Next, from the equation $\nabla\wedge{H}  +i\,\omega\varepsilon_{0}\varepsilon(x)\,{E}=0$, we have that
$$
{\rm div}(\varepsilon E)=0,
$$
with boundary condition $E\wedge\boldsymbol{\nu}=0$. As the argument in {\bf Step 2} of the proof of Proposition \ref{thm:2}, we introduce the transformation of coordinates such that $\boldsymbol{\nu}=e_3$ in the new coordinates $(y_1,y_2,y_3)$, and $\tilde{E}$ satisfies that ${\rm div}(\tilde{\varepsilon} \tilde{E})=0$, with boundary conditions $\tilde{E}\wedge e_3=0$ on the boundary $y_3=0$ in $B_R(x_0)$. The boundary condition is actually that
$$
\tilde{E}_1=\tilde{E}_2=0.
$$
We first work on the estimates of $\tilde{E}_1$, and the estimate of $\tilde{E}_2$ will be derived similarly.  By the equation,
$$
0=\partial_{y_1}({\rm div}(\tilde{\varepsilon}\tilde{E})) ={\rm div}(\partial_{y_1}(\tilde{\varepsilon})\tilde{E}+\tilde{\varepsilon}\partial_{y_1}(\tilde{E}))
$$
Note that
$$
\partial_{y_1}\tilde{E}_2=\partial_{y_2}\tilde{E}_1+(\nabla\wedge\tilde{E})_3,\qquad\mbox{and}\qquad \partial_{y_1}\tilde{E}_3=\partial_{y_3}\tilde{E}_1-(\nabla\wedge\tilde{E})_2.
$$
Let us define a vector $b$ as the following 
$$
b:=(0,(\nabla\wedge\tilde{E})_3, -(\nabla\wedge\tilde{E})_2)^T,
$$
then $\tilde{E}_1$ satisfies the following elliptic equation of second order,
$$
{\rm div}(\tilde{\varepsilon}\nabla\tilde{E}_1) =-{\rm div}(\partial_{y_1}(\tilde{\varepsilon})\tilde{E}+\tilde{\varepsilon}b),
$$
with the boundary condition $\tilde{E}_1=0$ on $\partial\Omega$.

Multiply $\overline{\eta^2\tilde{E}_1}$ on the both sides of the equation, and integrate by part in $B_R(y_0)$, where $y_0$ is the corresponding point of $x_0$ in the new coordinates, then
\begin{equation}
\int_{B_r\cap\Omega}|\nabla\tilde{E}_1|^2dy \le C(\lambda)\int_{B_R\cap\Omega}|\tilde{E}|^2|\nabla\tilde{\varepsilon}|^2{\rm d}y+ \frac{C(\lambda)}{R}\int_{B_R\cap\Omega}(|\tilde{E}|^2|\nabla\tilde{\varepsilon}|+||\tilde{\varepsilon}||_{\mathcal{L}^{\infty}}^2|b|^2)dy,
\end{equation}
where the constant $C(\lambda)$ only depends on the elliptic ratio.
Since 
$$
\int_{B_R\cap\Omega}|\tilde{E}|^2|\nabla\tilde{\varepsilon}|^2{\rm d}y\leq ||\nabla\tilde{\varepsilon}||_{\mathcal{L}^{3+\delta}}^2\left(\int_{B_R\cap\Omega}|\tilde{E}|^{\frac{6+2\delta}{1+\delta}}{\rm d}y\right)^{\frac{1+\delta}{3+\delta}},
$$
and
$$
\left(\int_{B_R\cap\Omega}|\tilde{E}|^{\frac{6+2\delta}{1+\delta}}{\rm d}y\right)^{\frac{1+\delta}{3+\delta}}\leq||\tilde{E}||_{\mathcal{L}^6(B_R\cap\Omega)}^2|B_R\cap\Omega|^{\frac{2\delta}{9+3\delta}}\leq CR^{\frac{2\delta}{3+\delta}}||\nabla \tilde{E}||_{\mathcal{L}^2(B_R\cap\Omega)}^2.
$$
Secondly,
$$
\int_{B_R\cap\Omega}|\tilde{E}|^2|\nabla\tilde{\varepsilon}|{\rm d}y\leq ||\nabla\tilde{\varepsilon}||_{\mathcal{L}^{3+\delta}}\left(\int_{B_R\cap\Omega}|\tilde{E}|^{\frac{6+2\delta}{2+\delta}}{\rm d}y\right)^{\frac{2+\delta}{3+\delta}},
$$
with
$$
\left(\int_{B_R\cap\Omega}|\tilde{E}|^{\frac{6+2\delta}{2+\delta}}{\rm d}y\right)^{\frac{2+\delta}{3+\delta}}\leq||\tilde{E}||_{\mathcal{L}^6(B_R\cap\Omega)}^2|B_R\cap\Omega|^{\frac{3+2\delta}{9+3\delta}}\leq CR^{\frac{3+2\delta}{3+\delta}}||\nabla \tilde{E}||_{\mathcal{L}^2(B_R\cap\Omega)}^2.
$$
Moreover,
$$
\int_{B_R\cap\Omega}|b|^2{\rm d}y\leq\int_{B_R\cap\Omega}|\nabla\wedge\tilde{E}|^{2}{\rm d}y.
$$
Hence by the above inequalities, we have
\begin{align*}
\int_{B_r\cap\Omega}|\nabla \tilde{E}_1|^2{\rm d}y\leq& C(\lambda)(R^{\frac{2\delta}{3+\delta}}||\nabla\tilde{\varepsilon}||^2_{\mathcal{L}^{3+\delta}}+R^{\frac{3+2\delta}{3+\delta}}||\nabla\tilde{\varepsilon}||_{\mathcal{L}^{3+\delta}})\int_{B_R\cap\Omega}|\nabla \tilde{E}|^2{\rm d}y\\
   &+C(R)||\tilde{\varepsilon}||^2_{\mathcal{L}^{\infty}}\int_{B_R\cap\Omega}|\nabla\wedge\tilde{E}|^2{\rm d}y+C(R)\int_{B_R\cap\Omega}|\tilde{E}|^2{\rm d}y.
%\leq& CR^{\frac{2\delta}{3+3\delta}}\int_{B_r\cap\Omega}|\nabla E|^2{\rm d}x+C(R)\int_{B_R\cap\Omega}|E|^2dx.
\end{align*}
Obviously $\tilde{E}_2$ satisfies the same estimate.

Finally, we are going to estimate $\tilde{E}_3$. Note that
$$
\partial_{y_1}\tilde{E}_3=\partial_{y_3}\tilde{E}_1-(\nabla\wedge\tilde{E})_2,\qquad\mbox{and}\qquad\partial_{y_2}\tilde{E}_3=\partial_{y_3}\tilde{E}_2+(\nabla\wedge\tilde{E})_1,
$$
and by the equation ${\rm div}(\tilde{\varepsilon}\tilde{E})=0$,
$$
\tilde{\varepsilon}_{3,3}\partial_{y_3}\tilde{E}_3=(-\sum_{i,j=1}^{3}\tilde{\varepsilon}_{ij}\partial_{y_i}\tilde{E}+\tilde{\varepsilon}_{3,3}\partial_{y_3}\tilde{E}_3)-\sum_{i,j=1}^3\partial_{y_i}\tilde{\varepsilon}_{ij}\tilde{E}_j,
$$
where $\tilde{\varepsilon}_{i,j}$ is the $(i,j)^{th}$ entry of matrix $\tilde{\varepsilon}$, and $\tilde{\varepsilon}_{3,3}\geq\lambda>0$.
So
\begin{align*}
&\int_{B_r\cap\Omega}|\nabla \tilde{E}_3|^2{\rm d}y\\
\leq& C(\lambda)\int_{B_r\cap\Omega}|\tilde{E}|^2|\nabla\tilde{\varepsilon}|^2{\rm d}y+C(\lambda)||\tilde{\varepsilon}||^2_{\mathcal{L}^{\infty}}\int_{B_r\cap\Omega}(|\nabla\wedge\tilde{E}|^2+|\nabla\tilde{E}_1|^2+|\nabla\tilde{E}_2|^2){\rm d}y.
\end{align*}
Then
\begin{align*}
\int_{B_r\cap\Omega}|\nabla \tilde{E}|^2{\rm d}y\leq& C_1(R^{\frac{2\delta}{3+\delta}}+R^{\frac{3+2\delta}{3+\delta}})\int_{B_R\cap\Omega}|\nabla \tilde{E}|^2{\rm d}y\\
   &+C(R)||\tilde{\varepsilon}||^2_{\mathcal{L}^{\infty}}\int_{B_R\cap\Omega}|\nabla\wedge\tilde{E}|^2{\rm d}y+C(R)\int_{B_R\cap\Omega}|\tilde{E}|^2{\rm d}y,
\end{align*}
where the constant $C_1$ only depends on $||\tilde{\varepsilon}||_{\mathcal{W}^{1,3+\delta}}$ and the elliptic ration. 

Apply the estimates of $\nabla\wedge\tilde{E}$ which we obtained before, and after the change of the coordinates back, we have
\begin{equation}\label{3.61x}
\int_{B_r\cap\Omega}|\nabla E|^2{\rm d}x\leq C_1(R^{\frac{2\delta}{3+\delta}}+R^{\frac{3+2\delta}{3+\delta}})\int_{B_R\cap\Omega}|\nabla \tilde{E}|^2{\rm d}x+C(R)\int_{B_R\cap\Omega}|E|^2dx.
\end{equation}

Finally, let $\eta_i$ be cut-off functions which satisfies $\sum_i \eta_i\equiv 1$ and such that the set of all the subregions  $\Omega_i:=\{x\in\Omega\,;\,\eta_i(x)>0\}$ together is a finite cover of $\Omega_i$, with the property that $\rm{diam}\{\Omega_i\}\leq R$. Then
%Based on the interior and boundary estimate, we can prove the $\mathcal{H}^1$-regularity of $H$ as follows
\begin{align*}
||E||_{\mathcal{H}^{1}(\Omega)}
%\leq & ||\nabla(\nabla\varphi)||_{\mathcal{L}^2(\Omega)}+||\nabla(\nabla\wedge A)||_{\mathcal{L}^2(\Omega)}+\|H\|_{\mathcal{L}^2(\Omega)}\\
\leq & \sum_{i}||\eta_i\nabla E||_{\mathcal{L}^2(\Omega)}+||E||_{\mathcal{L}^2(\Omega)}\\%+||H||_{\mathcal{L}^2(\Omega)})\\
\leq &C(R^{\frac{2\delta}{3+\delta}}+R^{\frac{3+2\delta}{3+\delta}})\sum_i||\nabla E||_{\mathcal{L}^2(\Omega_i)}+C||E||_{\mathcal{L}^2(\Omega)}\\
\leq &C(R^{\frac{2\delta}{3+\delta}}+R^{\frac{3+2\delta}{3+\delta}})||\nabla E||_{\mathcal{L}^2(\Omega)}+C||E||_{\mathcal{L}^2(\Omega)}.
\end{align*}

Therefore, let $R>0$ sufficiently small, We complete the proof of the estimate \eqref{equ:3.56E}.

Similarly, for the remaining two estimates, the interior estimates of $E$ and {\underline the boundary estimate of $\nabla\wedge E$} can be achieved by exchanging the role of $\varepsilon$ and $\mu$. The estimate of ${\rm div}E$ near the boundary $\partial\Omega$ can be deduced by using the relation ${\rm div}(\varepsilon E)=0$ carefully as done in the proof of the estimate \eqref{equ:3.56E} ({\it i.e.}, from \eqref{3.60x} to \eqref{3.61x}), then all the three estimates yield the remaining two estimates of $E$ in the proposition.
\end{proof}
Using all the results obtained before, we can now address the estimate to both $E$ and $H$ as given in Theorem \ref{them:10a}.
\begin{proof}[Proof of Theorem \ref{them:10a}]
By the embedding theorem, $\mu(x)\in \mathcal{C}^{\beta}(\Omega)$ and $\varepsilon(x)\in \mathcal{C}^{\beta}(\Omega)$. Then the theorem follows from Theorem \ref{thm:15} and Proposition \ref{thm:E4}. Moreover, the subset $\Omega_h$ is determined by considering the $\mathcal{H}^1$-estimates of the two functions $W_E$ and $W_H$ together, where $W_E:=E-V_E$ and $W_H:=H-V_H$. $V_H$ is the $V$ constructed in the proof of Theorem \ref{thm:15}, while $V_E$ is similar but with respect to $E$. Clearly, we have that $\Omega_h$ is open and $meas(\Omega\backslash\Omega_h)=0$.
\end{proof}

\section{Appendix}
In the appendix, we will introduce several important functional spaces and properties which are used in this paper. Because most of them are well-known, we only list the lemmas and omit the proof for the shortness.
\begin{defn}[Morrey Space]\label{def1}
We denote by $\mathcal{L}^{p,\nu}(\Omega)$ the Morrey spaces, and for $1 \le p < \infty $ and $\nu \in (0,n+p)$, its norm is defined as
$$||u||_{\mathcal{L}^{p,\nu}}(\Omega)=\{\sup_{x \in \Omega, 0<r<d}r^{-\nu}\int_{\Omega(x,r)}|u(z)|^p dz\}^{\frac{1}{p}}$$
where $\Omega(x,r)=\Omega \cap B(x,r)$, $B(x;r)$ is a ball of which centre is $x \in \Omega$. We call $u(x)\in \mathcal{L}^{p,\nu}(\Omega)$ if $||u||_{\mathcal{L}^{p,\nu}}(\Omega) < \infty$.
\end{defn}

\begin{defn}[Campanato space]\label{def:14}
We denote by $\mathfrak{L}^{p,\nu}(\Omega)$ the Campanato spaces, and for the same $p, \nu$ used in the definition \ref{def1}, the norm of the Campanato spaces is given as
$$||u||_{\mathfrak{L}^{p,\nu}}(\Omega)=||u||_{\mathcal{L}^{p,\nu}}(\Omega)+\{\sup_{x \in \Omega, 0<r<d}r^{-\nu}\int_{\Omega(x,r)}|u(z)-\bar{u}_{x,r}|^p dz\}^{\frac{1}{p}}$$ where $\bar{u}_{x,r}=\frac{1}{|\Omega(x,r)|}\int_{\Omega(x,r)}u(z)dz$.
\end{defn}

The Morrey space and Campanato space have the following two well-known and important properties.

\begin{lem}\label{lem:16}
$\mathfrak{L}^{p,\nu}(\Omega)\cong \mathcal{C}^{0,\delta}$, where $\delta=\frac{\nu-n}{p}$ and $n<\nu<n+p$.
\end{lem}

\begin{lem}
$\mathcal{L}^{p,n}(\Omega)\cong \mathcal{L}^{\infty}$.
\end{lem}

\begin{lem}\label{lem:24a}
$\mathcal{L}^{p,\nu}(\Omega)\cong \mathfrak{L}^{p,\nu}$, for $0\leq\nu<n$.
\end{lem}

\begin{defn}[BMO space]
We denote by $\operatorname{BMO}(\Omega)$ the bounded mean oscillation function space, and its norm is defined as
$$||u||_{\operatorname{BMO}(\Omega)}=||u||_{\mathcal{L}^1(\Omega)}+\sup_{Q\subset\Omega}\frac{1}{|Q|}\int_{Q}|u-\bar{u}_Q| dx $$
where $Q$ is a hypercube in $\Omega$ and $|Q|$ denotes its volume. We call $u \in \operatorname{BMO}(\Omega) $ if $||u||_{\operatorname{BMO}(\Omega)}<\infty$.
\end{defn}

\begin{lem}
For any open bounded domain $\Omega$, we have that $||u||_{\operatorname{BMO}(\Omega)}\le ||u||_{\mathfrak{L}^{2,n}(\Omega)}$.
\end{lem}

By employing the BMO space, we have the following interpolation theorem.

\begin{lem}[Stampacchia interpolation theorem]
Let $1<q<+\infty$, and $T$ be a linear operator.
If $$\begin{aligned}
||Tu||_{\mathcal{L}^q(\Omega)}\le& C_1||u||_{\mathcal{L}^q(\Omega)}, \qquad \forall u \in \mathcal{L}^q(\Omega),\\
||Tu||_{\operatorname{BMO}(\Omega)}\le& C_2||u||_{\mathcal{L}^{\infty}(\Omega)}, \qquad \forall u \in \mathcal{L}^{\infty}(\Omega).
\end{aligned}
$$
Then for $p\in [q,+\infty)$, $||Tu||_{\mathcal{L}^p(\Omega)}\le C||u||_{\mathcal{L}^p}(\Omega)$.
\end{lem}

The final one is the Helmholtz decomposition in the vector analysis.

\begin{lem}[Helmholtz Decomposition] \label{lem16} Suppose $\Omega$ is a bounded, simply-connected, Lipschitz domain. Every square-integrable vector field $u\in(\mathcal{L}^2(\Omega))^3$ has an orthogonal decomposition:
$$
u=\nabla\varphi+\nabla\wedge A,
$$
where $\varphi\in \mathcal{H}^1(\Omega)$ is a scalar function, and $A\in \mathcal{H}^1(curl,\Omega)$.
\end{lem}

\begin{acknowledgements}
The authors would like to thank Prof. Yves Capdeboscq for the helpful discussion. 
\end{acknowledgements}



\end{document}